\documentclass[a4paper,11pt]{amsart}
\usepackage{amscd}
\usepackage{graphics}
\usepackage{url}
\usepackage{epsfig}
\usepackage{parskip}
\usepackage{amssymb}
\usepackage{enumerate}
\usepackage{color}
\usepackage[all]{xy}
\usepackage{overpic}

\usepackage[margin=1.1in]{geometry}

\newtheorem{thm}{Theorem}[section]
\newtheorem{corollary}[thm]{Corollary}
\newtheorem{lemma}[thm]{Lemma}
\newtheorem{prop}[thm]{Proposition}

\newtheorem{conjecture}[thm]{Conjecture}
\theoremstyle{definition}
\newtheorem{defn}[thm]{Definition}

\newtheorem{remark}[thm]{Remark}
\newtheorem{notation}[thm]{Notation}
\newtheorem{construction}[thm]{Construction}
\newtheorem{assumption}[thm]{Assumption}

\newcommand{\Z}{\mathbb{Z}}
\newcommand{\R}{\mathbb{R}}
\newcommand{\N}{\mathbb{N}}

\newcommand{\Hy}{\mathbb{H}}

\newcommand{\cA}{\mathcal{A}}
\newcommand{\cB}{\mathcal{B}}
\newcommand{\cC}{\mathcal{C}}

\newcommand{\cE}{\mathcal{E}}

\newcommand{\cG}{\mathcal{G}}

\newcommand{\cP}{\mathcal{P}}

\newcommand{\cX}{\mathcal{X}}

\newcommand{\CAT}{\operatorname{CAT}}

\DeclareMathOperator{\isom}{Isom}
\DeclareMathOperator{\stab}{Stab}
\DeclareMathOperator{\homeo}{Homeo}

\DeclareMathOperator{\lcm}{lcm}
\DeclareMathOperator{\Aut}{Aut}

\newcommand{\la}{\langle}
\newcommand{\ra}{\rangle}

\newcommand{\p}{\partial}

\newcommand{\G}{\Gamma}

\newcommand{\oB}{\overline{\p \cB_i}}
\newcommand{\bp}{\overline{p}}
\newcommand{\wZ}{\widetilde{Z}}
\newcommand{\tY}{\widetilde{Y}}

\newcommand{\wt}{\widetilde{\theta}}
\newcommand{\wz}{\tilde{z}}

\definecolor{amethyst}{rgb}{0.6, 0.4, 0.8}

\newcommand{\hide}[1]{}

\title{Quasi-isometric groups with no common model geometry }
\author{Emily Stark, Daniel Woodhouse }
\date{ \today }

\address{Department of Mathematics\\
 Technion - Israel Institute of Technology \\
 Haifa 32000 \\
  Israel }
\email{emily.stark@technion.ac.il, woodhouse.da@technion.ac.il}

\thanks{The first author was supported by the Azrieli Foundation and the Israel Science Foundation (grant 1941/14) and was supported in part at the Technion by a Zuckerman STEM Leadership Fellowship. The second author was supported by the Israel Science Foundation (grant 1026/15).}

\begin{document}

\maketitle

\begin{abstract}
  A \emph{simple surface amalgam} is the union of a finite collection of surfaces with precisely one boundary component each and which have their boundary curves identified.
  We prove that if two fundamental groups of simple surface amalgams act properly and cocompactly by isometries on the same proper geodesic metric space, then the groups are commensurable.
  Consequently, there are infinitely many fundamental groups of simple surface amalgams that are quasi-isometric, but which do not act properly and cocompactly on the same proper geodesic metric space.
\end{abstract}

\section{Introduction}

 A {\it model geometry} for a group is a proper geodesic metric space on which the group acts properly and cocompactly by isometries. Since a finitely generated group is quasi-isometric to any of its model geometries, if two groups have a common model geometry, then the two groups are quasi-isometric. An interesting problem in geometric group theory is to determine for which classes of groups the converse holds. 

  Every group quasi-isometric to hyperbolic $n$-space $\Hy^n$ acts properly and cocompactly by isometries on $\Hy^n$ \cite{tukia,gabai,cassonjungreis,tukia86,tukia94}. There is a dichotomy between groups quasi-isometric to $\Hy^2$ and $\Hy^n$ for $n \geq 3$: all groups quasi-isometric to $\Hy^2$ are commensurable, but there are infinitely many commensurability classes among groups quasi-isometric to $\Hy^n$ for $n \geq 3$. Among virtually non-cyclic free groups, there is one quasi-isometry class and one abstract commensurability class, yet, two such groups do not necessarily share a common model geometry. Mosher--Sageev--Whyte \cite{moshersageevwhyteI} prove if $p,q>2$ are prime, then $\Z/p\Z * \Z/p\Z$ and $\Z/q\Z * \Z / q\Z$ do not have a common model geometry unless $p=q$. 
   
   In this paper, we exhibit an infinite family of one-ended torsion-free hyperbolic groups that contains one quasi-isometry class, infinitely many abstract commensurability classes, and whose members have a common model geometry if and only if they are commensurable. Let $\mathcal{Y}_k$ be the set of spaces homeomorphic to the union of $k \geq 3$ surfaces with one boundary component identified to each other along their boundary curves. Let $\mathcal{C}_k$ be the set of fundamental groups of spaces in $\mathcal{Y}_k$. The main result of this paper is the following. 

  \begin{thm} \label{maintheorem}
    Let $G, G' \in \mathcal{C}_k$. The following are equivalent.
      \begin{enumerate}
	 \item The groups $G$ and $G'$ act properly and cocompactly by isometries on the same proper geodesic metric space.
	 \item The groups $G$ and $G'$ are abstractly commensurable.
	 \item There exists a group $\mathfrak{G}$ that contains $G$ and $G'$ as finite-index subgroups.
      \end{enumerate}
  \end{thm}

  Theorem \ref{maintheorem} together with the following results proves there are infinitely many groups in $\mathcal{C}_k$ that are quasi-isometric but do not act properly and cocompactly on the same proper geodesic metric space. For each $k \geq 3$, all groups in $\mathcal{C}_k$ are quasi-isometric. Indeed, using the JSJ tree of Bowditch \cite{bowditch} and the techniques of Behrstock--Neumann \cite{behrstockneumann}, Malone proved that if $G \in \mathcal{C}_k$ and $G' \in \mathcal{C}_{\ell}$, then $G$ and $G'$ are quasi-isometric if and only if $k = \ell$ \cite[Theorem 4.14]{malone}. On the other hand, there are infinitely many abstract commensurability classes within $\mathcal{C}_k$. The abstract commensurability classification within $\cC_k$ was given by the first author \cite{stark} for $k=4$ and easily extends to arbitrary $k$; see also \cite{danistarkthomas}.  
  
  The main result of this paper is an example of rigidity that may manifest within a class of groups that is not quasi-isometrically rigid. (See Dru{\c{t}}u--Kapovich \cite{drutukapovich} for background on quasi-isometric rigidity.) A class of groups exhibits {\it action rigidity} if whenever two groups in the class act properly and cocompactly on the same proper geodesic metric space, the groups are (abstractly) commensurable. This form of rigidity fails in general. For example, Burger--Mozes \cite{burgermozes} provide examples of simple groups which act geometrically on the product of two infinite trees. Nonetheless, this phenomenon is intriguing to explore within other classes of finitely generated groups for which the quasi-isometry and commensurability classifications do not coincide.

  \subsection{Method of Proof}
  
   The main aim of the paper is to prove Condition (1) implies Condition (3): if two groups $G,G' \in\cC_k$ have a common model geometry, then there is a group which contains both $G$ and $G'$ as finite-index subgroups. To accomplish this goal, we prove in Theorem~\ref{thm:the_isomorphism} if $G$ and $G'$ have a common model geometry, then there exists a particularly nice model geometry for $G$ and $G'$: a $2$-dimensional $\CAT(0)$ cube complex $\cX$. 
    
   The construction of the square complex $\cX$ relies heavily on the structure of the visual boundary of $G$ and $G'$ and the corresponding JSJ decompositions of these groups given by Bowditch \cite{bowditch}. In particular, if $G$ and $G'$ have a common model geometry $X$, then the visual boundaries of $G$ and $G'$ are homeomorphic to the visual boundary of $X$. Bowditch constructs the JSJ tree of a one-ended hyperbolic group $G$ using the topology of the visual boundary of $G$, where the {\it JSJ tree} of $G$ is the Bass-Serre tree of the JSJ decomposition of $G$. Thus, if $G$ and $G'$ have a common model geometry, then $G$ and $G'$ act by isometries on the same JSJ tree $T$. 
   
   We use the actions of $G$ and $G'$ on the JSJ tree $T$ and the notion of weak convex hull of Swenson \cite{swenson} to define a ``coarse graph of spaces'' decomposition of $X$ which is preserved by $G$ and $G'$. The work of Mosher--Sageev--Whyte \cite{moshersageevwhyteI} allows us to replace the ``vertex spaces'' of the coarse graph of spaces with locally-finite bushy trees on which the vertex groups of $G$ and $G'$ act properly and cocompactly by isometries.
   We then analyze the action of vertex groups of $G$ and $G'$ on the bushy trees. In particular, we prove in Proposition~\ref{prop:comm_tr_len} that   
   distinguished elements in $G$ and $G'$ act on lines in the trees with translation lengths which satisfy a certain commensurability condition. This condition allows us to glue the trees together using strips to build a CAT$(0)$ square complex $\cX$ on which $G$ and $G'$ act properly and cocompactly by isometries. 
   
   The automorphism group of the cube complex $\mathcal{X}$ could be uncountable; the groups $G$ and $G'$ need not be finite-index subgroups of this automorphism group. 
   We argue there exist actions of $G$ and $G'$ on $\mathcal{X}$ that induce isometries of $T$ which preserve a \emph{rigid edge coloring} of~$T$.
   We prove in Theorem~\ref{thm:NewAction} the group of automorphisms of $\mathcal{X}$ that induce a color-preserving isometry of $T$ acts properly on $\cX$ and contains both $G$ and $G'$ as finite-index subgroups.
    
   The other implications in Theorem~\ref{maintheorem} follow from previous work. By \cite{stark,danistarkthomas}, two groups in $\cC_k$ are abstractly commensurable if and only if the {\it Euler characteristic vectors} of the two groups are commensurable vectors. (For a precise statement, see Section~\ref{sec:main_thm}.) The author shows there is a minimal element within $\cC_k$ within each abstract commensurability class, hence Condition~(2) and Condition~(3) are equivalent, and Condition~(3) implies Condition~(1). 
   
    \begin{figure}
	\begin{overpic}[scale=.25,tics=10]{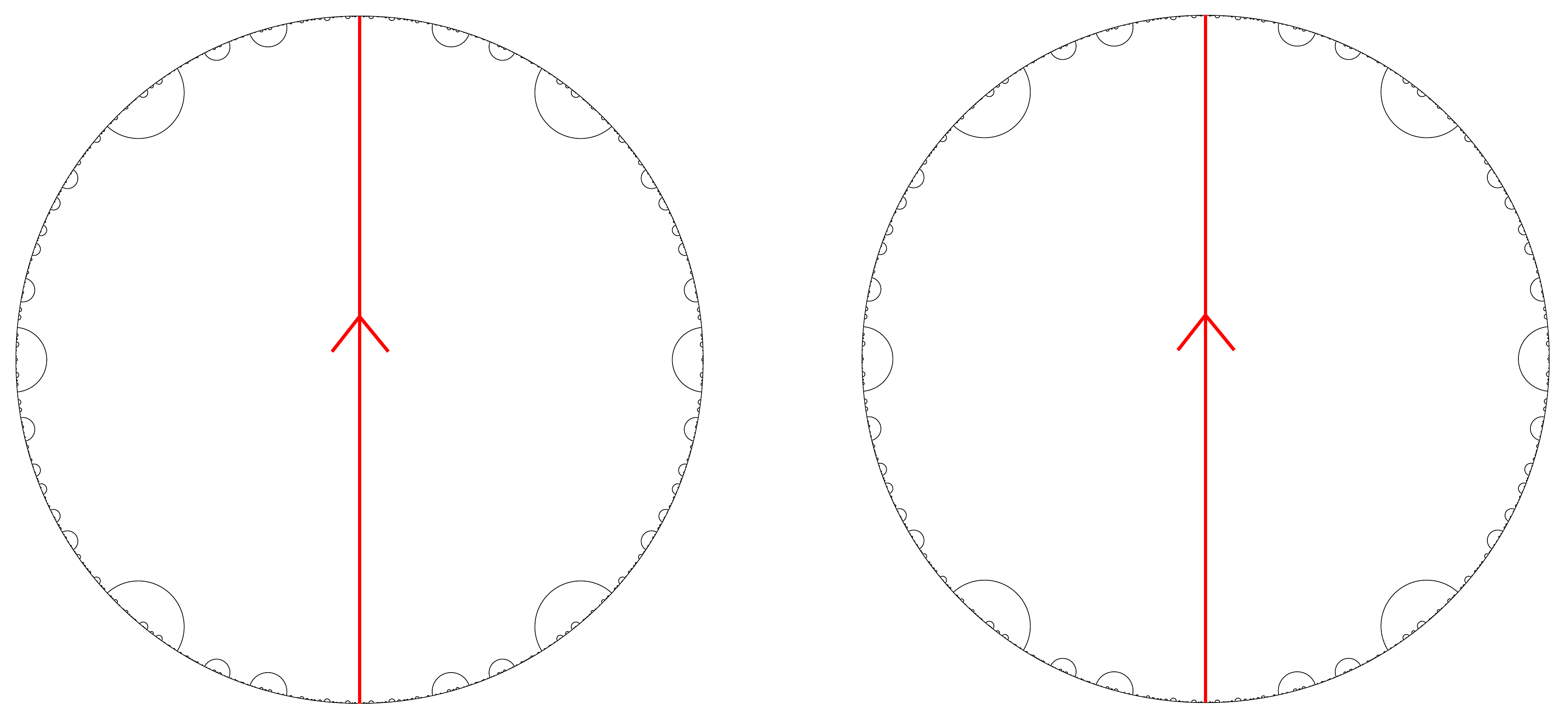} 
	 \end{overpic}
		 \caption{ Suppose $Y$ is the union of four surfaces with genus one and one boundary component glued along their boundary curves. A simple example of a model geometry for $\pi_1(Y) \in \cC_4$ consists of copies of the hyperbolic plane glued along geodesic lines; two copies are drawn above. In this case, the hyperbolic metrics can be chosen so that $\isom(\tY)$ does not act properly on $\tY$. The image was drawn with lim \cite{mcmullen}.}
	 \label{figure:hyp_plane}
   \end{figure}    
  
    \subsection{Generalizations}
    
    It is natural to ask if our results extend from $\cC_k$ to its quasi-isometry class, and, more generally, to the class $\cC$ of one-ended hyperbolic groups that are not Fuchsian and whose JSJ decomposition over two-ended subgroups contains only two-ended and maximally hanging Fuchsian vertex groups. Indeed, many of the techniques of this paper are useful in pursuing these more general results. For example, we believe if $G, G' \in \cC$ have a common model geometry, then the results in Section~\ref{sec:cubulation} can be generalized to produce geometric actions on a common $\CAT(0)$ square complex $\cX$. However, it is significantly harder to generalize the full strength of Theorem~\ref{thm:the_isomorphism}: the existence of actions that preserve a rigid edge coloring of the Bass-Serre tree cannot be deduced from our techniques in the case that $G$ and $G'$ split over non-isomorphic graphs or in the case $G$ and $G'$ split over a suitably complicated graph. 
    We believe that deducing commensurability from the existence of a common model geometry requires the following generalization of Leighton's graph covering theorem~\cite{leighton} (see also~\cite{basskulkarni,woodhouse18}), first asked by Haglund~\cite[Problem 2.4]{haglund06}. 
    
   \begin{conjecture}
    Let $X$ and $X'$ be special cube complexes with isomorphic universal covers. 
    Then, there exist finite-sheeted covers $\hat{X} \rightarrow X$ and $\hat{X}' \rightarrow X$ such that $\hat{X} \cong \hat{X}'$.
   \end{conjecture}

   Haglund~\cite{haglund06} provided a positive result for uniform lattices in regular right-angled Fuchsian buildings, provided the chamber is a polygon with at least six edges, and subsequently, he provided a positive result for right-angled buildings associated to graph products of finite groups~\cite{haglund08}.
  
    In addition, based on results by Dani--Stark--Thomas \cite{danistarkthomas}, we conjecture Conditions~(2) and (3) are not equivalent in this generality.  Finally, as explained in Remark~\ref{remark:RACG}, the conclusion of our main theorem extends to a related class of right-angled Coxeter groups.

    \subsection*{Acknowledgements}\label{ackref}
    The authors are thankful for helpful discussions with Michah Sageev. The authors thank the anonymous referee for useful comments.

  
\section{Preliminaries} \label{sec:prelims}

\subsection{Metric notions and group actions}

  \begin{defn}
    Let $(X,d_X)$ and $(Y,d_Y)$ be metric spaces. A map $\phi$ from $X$ to $Y$ is a \emph{$(K,\epsilon)$-quasi-isometry} if there are constants $K\geq 1$ and $\epsilon \geq 0$ such that the following hold: 
      \begin{enumerate}
	\item The map $\phi$ is a {\it $(K,\epsilon)$-quasi-isometric embedding}: for all $x_1, x_2 \in X$, \[\frac{1}{K}\,d_X(x_1,x_2)-\epsilon \leq d_Y\bigl(\phi(x_1),\phi(x_2)\bigr)\leq K\,d_X(x_1,x_2)+\epsilon.\]
	\item The map $\phi$ is {\it $\epsilon$-quasi-surjective}: every point of $Y$ lies in the $\epsilon$-neighborhood of the image of $f$. 
      \end{enumerate}
  \end{defn}

  \begin{defn}
    A subspace $C$ of a metric space $X$ is said to be {\it quasi-convex} if there exists a constant $k>0$ such that for all $x,y \in C$, each geodesic joining $x$ to $y$ is contained in the $k$-neighborhood of $C$. Let $G$ be a finitely generated group and $d$ a word metric on $G$ defined with respect to a finite generating set. A subgroup $H \leq G$ is {\it quasi-convex} if there exists a constant $k>0$ such that for all $h,h' \in H$, each geodesic joining $h$ to $h'$ in $G$ is contained in the $k$-neighborhood of $H$ in $G$.
  \end{defn}

  \begin{notation}
   Suppose $A$ is a subspace of a metric space $X$ and $R \in \R^+$. We use $N_R(A)$ to denote the $R$-neighborhood of $A$ in $X$. If $Z$ is a topological space and $z \in Z$, the pair $(Z,z)$ is a {\it pointed topological space}. 
  \end{notation}
  
  \begin{defn} (Group actions.)
    An {\it action} of a group $G$ on a topological space $X$ is a homomorphism $\Phi:G \rightarrow \homeo(X)$, where $\homeo(X)$ is the group of homeomorphisms of $X$. An {\it action by isometries} of a group $G$ on a metric space $X$ is a homomorphism $\Phi:G \rightarrow \isom(X)$, where $\isom(X)$ is the group of isometries of $X$. For $x \subset X$, we write $g \cdot x$ for the image of $x$ under $\Phi(g)$ and $G \cdot x$ for $\{ g \cdot x \mid g \in G\}$. An action $\Phi$ is {\it faithful} if $\ker(\Phi) = \{1\}$. An action of a group $G$ on a space $X$ is {\it cocompact} if there exists a compact set $K \subset X$ such that $X = G \cdot K$. If $X$ is a metric space, an action of $G$ on $X$ is {\it proper} if for each $x \in X$, there is a number $\epsilon>0$ so that the set $\{g \in G \, | \, d(x,g\cdot x) \leq \epsilon \}$ is finite.  
  \end{defn}

  We will make use of the following elementary lemma, which follows easily from standard techniques; see \cite{bowditch06,bridsonhaefliger}. 

  \begin{lemma}\label{lemma:subgroup}
  Let $X$ be a proper metric space, and let $G$ be a group which acts properly on~$X$.
  If $H \leqslant G$ acts cocompactly on $X$, then $G$ acts on $X$ cocompactly, and $H$ is a finite-index subgroup of $G$.
  \end{lemma}

\subsection{The JSJ decomposition}

  The groups considered in this paper are fundamental groups of finite graphs of groups; for background, see \cite{scottwall}, \cite{serre}. We use the following notation. 

  \begin{defn} \label{def:carry} 
   A \emph{graph of groups} $\mathcal{G}$ is a graph $\Gamma = (V\Gamma, E\Gamma)$ with a \emph{vertex group} $G_v$ for each $v \in V\Gamma$, an \emph{edge group} $G_e$ for each $e \in E\Gamma$, and \emph{edge maps}, which are injective homomorphisms $\Theta^{\pm}_e: G_e \rightarrow G_{\pm e}$ for each $e =(-e,+e) \in E\Gamma$.

          \begin{figure}
   \begin{overpic}[width=.7\textwidth,tics=10]{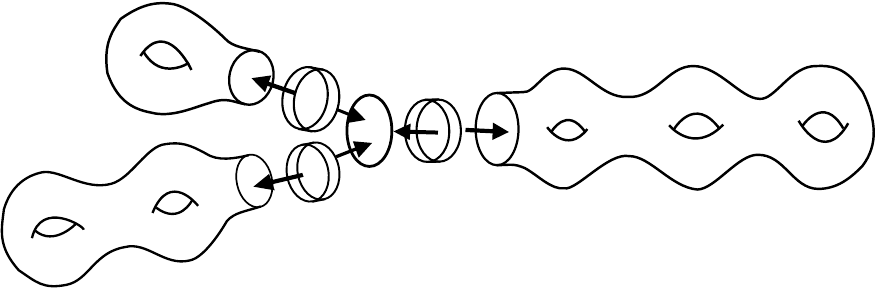} 
    \put(7,25){$Y_{v_1}$}
    \put(-5,5){$Y_{v_2}$}
    \put(80,7){$Y_{v_3}$}
    \put(32,27){$Y_{e_1} \times I$}
    \put(32,7){$Y_{e_2} \times I$}
    \put(41,23){$Y_{v_0}$}
    \put(48,23){$Y_{e_3} \times I$}
   \end{overpic}
    \caption{The graph of spaces decomposition for $G$. The black arrows indicate the attaching maps and $I$ denotes the interval $[-1,1]$.} 
    \label{figure:gr_spaces}
    \end{figure}
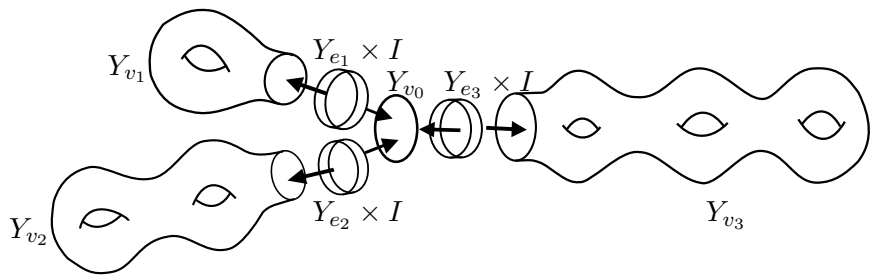
   
   A \emph{graph of spaces} associated to a graph of groups $\mathcal{G}$ is a space $Y$ constructed from a pointed \emph{vertex space} $(Y_v, y_v)$ for each $v \in V\Gamma$ with $\pi_1(Y_v,y_v) = G_v$, a pointed {\it edge space} $(Y_e, y_e)$ for each $e = (-e,+e) \in E\Gamma$ such that $\pi_1(Y_e,y_e) = G_e$, and maps $\theta^{\pm}_e: (Y_e,y_e) \rightarrow (Y_{\pm e}, y_{\pm{e}})$ such that $(\theta^{\pm}_e)_* = \Theta^{\pm}_e$.
   The space $Y$ is $$\left(\bigsqcup_{v \in V\Gamma} Y_v \bigsqcup_{e \in E\Gamma} \left( Y_e \times [-1,1]\right)\right) \; \Big/ \; \{(y,\pm 1) \sim \theta_e^{\pm}(y) \mid (y, \pm 1) \in Y_e\times [-1,1] \}. $$
   The \emph{fundamental group} of the graph of groups $\mathcal{G}$ is $\pi_1(Y)$. The {\it underlying graph} of the graph of groups $\cG$ is the graph $\G$.
   A group $G$ \emph{splits as graph of groups} if $G$ is the fundamental group of a non-trivial graph of groups.
    \end{defn}
   
    \begin{defn} \cite{bowditch}
      A {\it bounded Fuchsian group} is a Fuchsian group that is convex cocompact but not cocompact; that is, the group acts cocompactly on the convex hull of its limit set in $\p \Hy^2$, but does not act cocompactly on $\Hy^2$. The convex core of the quotient is a compact orbifold with non-empty boundary consisting of a disjoint union of compact $1$-orbifolds. The {\it peripheral subgroups} are the maximal two-ended subgroups which project to the fundamental groups of the boundary $1$-orbifolds.  A {\it hanging Fuchsian} subgroup $H$ of a group is a virtually-free quasiconvex subgroup together with a collection of {\it peripheral} two-ended subgroups, which arise from an isomorphism of $H$ with a bounded Fuchsian group. A {\it full quasiconvex subgroup} of a group $G$ is a subgroup that is not a finite-index subgroup of any strictly larger subgroup of $G$.
    \end{defn}

    \begin{thm}\cite[Thm 0.1]{bowditch}
      Let $G$ be a one-ended hyperbolic group that is not Fuchsian. There is a canonical \emph{JSJ decomposition} of $G$ as the fundamental group of a graph of groups such that each edge group is 2-ended and each vertex group is either (1)  $2$-ended; (2) maximal hanging Fuchsian; or, (3) a maximal quasi-convex subgroup not of type (2). These types are mutually exclusive, and no two vertices of the same type are adjacent. Every vertex group is a full quasi-convex subgroup. Moreover, the edge groups that connect to any given vertex group of type (2) are precisely the peripheral subgroups of that group.      
    \end{thm} 
    
   \begin{remark} \label{remark:gr_of_gps} (JSJ decomposition of $G,G' \in \cC_k$.)
    We will use the notation set in this remark throughout the paper. 
    Let $G, G' \in \cC_k$. The JSJ decompositions of $G$ and $G'$ are graphs of groups $\cG$ and $\cG'$, respectively, which are described as follows. The graph $\G$ associated to the graph of groups $\cG$ has vertex set $V\Gamma = \{u_0, \ldots, u_k\}$ and edge set $E\Gamma = \{(u_0, u_i) \mid 1 \leq i \leq k\}$. The vertex groups are $G_0 = G_{u_0} =  \langle g_0 \rangle \cong \mathbb{Z}$ and $G_i = G_{u_i} = \pi_1(\Sigma_i)$ for $1 \leq i \leq k$, where $\Sigma_i$ is an orientable surface with negative Euler characteristic and one boundary component. Let $g_i \in G_i$ be the homotopy class of the boundary curve of $\Sigma_i$ with some choice of orientation.
    If $e_i = (u_0,u_i) \in E\G$, then the edge group $G_{e_i}$ is $\langle \overline{g}_i\rangle \cong \mathbb{Z}$. The edge map $\Theta_e^-: G_e \rightarrow G_0$ is given by the map $\overline{g}_i \mapsto g_0$ and the edge map $\Theta_e^+: G_e \rightarrow G_i$ is given by the map $\overline{g}_i \mapsto g_i$.
   
    A graph of spaces $Y$ associated to the graph of groups $\cG$ has vertex spaces $Y_{u_0} \cong S^1$ and $\{ Y_{u_i} \cong \Sigma_i \,|\, 1 \leq i \leq k\}$. If $e = (u_0,u_i) \in E\G$, then $Y_e \cong S^1$. The attaching map $\theta_e^-: Y_e \rightarrow Y_{u_0}$ is a homeomorphism, and the attaching map $\theta_e^+:Y_e \rightarrow Y_{u_i}$ is a homeomorphism from $Y_e$ onto the boundary curve of $\Sigma_i$. An example appears in Figure~2.
   
    The JSJ tree $T$ of $G$ is an infinite biregular tree with vertices of valence $k$ and vertices of valence the cardinality of $\N$. The finite-valence vertices of $T$ are exactly those stabilized by conjugates of the $2$-ended vertex group $G_0$ in the JSJ decomposition of $G$, and the infinite-valence vertices of $T$ are exactly those stabilized by conjugates of a maximal hanging-Fuchsian subgroup $G_i = \pi_1(\Sigma_i)$ of~$G$.
    
    The JSJ decomposition $\cG'$ of $G'$ is similar; the underlying graph is the graph $\Gamma$. The vertex groups of $\cG'$ are $G_0' = G'_{u_0} = \la g_0' \ra \cong \Z$ and $G_i' = G'_{u_i} = \pi_1(\Sigma_i')$ for $1 \leq i \leq k$, where $\Sigma_i'$ is an orientable surface with negative Euler characteristic and one boundary component. Let $g_i' \in G_i'$ be the homotopy class of the boundary curve of $\Sigma_i'$ with some choice of orientation. If $e_i = (u_0,u_i) \in E\G$, then the edge group $G'_{e_i}$ is $\langle \overline{g}'_i\rangle \cong \mathbb{Z}$. The edge map $\Theta_{e_i}^{'-}: G_{e_i}' \rightarrow G_0'$ is given by the map $\overline{g}_i' \mapsto g_0'$ and the edge map $\Theta_{e_i}^{'+}: G_{e_i}' \rightarrow G_i'$ is given by the map $\overline{g}_i' \mapsto g_i'$. A graph of spaces $Y'$ associated to $\cG'$ is analogous to the space $Y$. 
    \end{remark}
    
    An example of a model geometry for a group in $\cC_k$ appears in Figure~1.

    \begin{notation}
      If $X$ is a $\delta$-hyperbolic metric space, let $\partial X$ denote the visual boundary of~$X$. Similarly, if $G$ is a $\delta$-hyperbolic group, let $\partial G$ denote the visual boundary of $G$. For the definition of visual boundary and its basic properties, see \cite{bridsonhaefliger} and the survey \cite{kapovichbenakli}. 
    \end{notation}
    
  \begin{remark} \label{remark:same_JSJ_tree} (Isomorphic JSJ trees for groups which have a common model geometry.) 
    The construction of the JSJ tree of a $\delta$-hyperbolic group given by Bowditch uses only the topology of $\p G$. (See Section 3 in~\cite{bowditch}.) Thus, if $G, G' \in \cC_k$ act properly and compactly on a proper geodesic metric space $X$, then, since both $\partial G$ and  $\partial G'$ are homeomorphic to $\partial X$, the groups $G$ and $G'$ have isomorphic JSJ trees, which we denote by $T$. Moreover, any element of $\isom(X)$ induces a homeomorphism of $\partial X$ and, therefore, yields an isometry of~$T$. Thus, there is a homomorphism $\Phi_T: \isom(X) \rightarrow \isom(T)$ that when composed with a homomorphism $\Phi:G \rightarrow \isom(X)$ or $\Phi': G' \rightarrow \isom(X)$ gives the action of $G$ or $G'$ on $T$ by isometries.
    \end{remark}
    
    We will make use of the following details of the cut point structure of $\partial G$. 

  \begin{defn} \label{cyclicordering}
    A set $\Delta$ is \emph{cyclically ordered} if there is a quaternary relation $\delta$ on $\Delta$ such that for every finite subset $F \subseteq \Delta$ there is an embedding $F \hookrightarrow S^1$ such that $\delta(w,x,y,z)$ holds for $w,x,y,z \in F$ if $x, z$ are contained in distinct components of $S^1 - \{w,y\}$.
  \end{defn}

  \begin{prop} \cite{bowditch} \label{prop:bow_cyclic_order}
    \begin{enumerate}
      \item If $v \in VT$ such that $G_v$ is a 2-ended subgroup, then $\partial G_v \subseteq \partial G$ is a unique pair of cut points that separates $\partial G$ into $n$ components, where $n$ is the valence of both local cut points $\{x, y\} = \partial G_v$. 
      \item If $v \in VT$ is a maximal hanging Fuchsian subgroup, then $\partial G_v \subseteq \partial G$ is a cyclically ordered Cantor set. The cyclic order is determined by the topology of $\partial G$: the relation $\delta(w,x,y,z)$ holds if $x, z$ are in distinct components of $\partial G - \{w,y\}$.
    \end{enumerate}
  \end{prop}
  
  \begin{defn} \label{def:op}
    Let $\Delta$ be a Cantor set with a fixed cyclic order. 
    Let $\gamma$ be a homeomorphism of $\Delta$ that preserves the cyclic order.
    The set $\Delta$ may be embedded into $S^1$ consistent with the cyclic ordering. Fix an orientation on $S^1$, which determines an orientation on $\Delta$. The homeomorphism $\gamma$ can be extended to a homeomorphism of $S^1$ by linearly extending the action across the missing intervals.
    We say $\gamma$ is \emph{orientation preserving} if this extended action preserves the orientation of $S^1$.
    This property of $\gamma$ does not depend on the choice of the embedding of $\Delta$ into $S^1$.
  \end{defn}

  The vertex and edge groups of the JSJ decomposition of $G \in \cC_k$ are quasi-convex, and hence satisfy the {\it bounded packing property}, introduced by Hruska--Wise \cite{hruskawise}.
  
  \begin{defn}
   Let $G$ be a finitely generated group and $d$ a word metric on $G$ defined with respect to a finite generating set. A subgroup $H$ of $G$ has {\it bounded packing in $G$} if for each constant $D$, there is a number $N =N(G,H,D) \geq 2$ so that for any collection of $N$ distinct cosets $gH$ in $G$, at least two are separated by a distance of at least $D$. Similarly, a collection of subsets $\cA = \{A_i\}_{i \in I}$ of a metric space $X$ has {\it bounded packing in $X$} if for each constant $D$, there exists $N = N(D) \geq 2$ so that for any collection of $N$ distinct subspaces $A_i \subset \cA$, at least two are separated by a distance of at least $D$. 
    \end{defn}

  \begin{lemma} \label{lemma:bounded_packing} \cite{gitikmitraripssageev}\cite[Section 4]{hruskawise}
    If $G \in \cC_k$, then the vertex groups of the JSJ decomposition of $G$ have bounded packing in $G$. 
  \end{lemma}
 
  Finally, we will use the classification of isometries of a hyperbolic proper geodesic metric space. See \cite[Proposition 4.1]{kapovichbenakli} for background.

  \begin{prop}  \label{prop:elliptic}
        Let $X$ be a proper, hyperbolic geodesic metric space and let $\gamma : X \rightarrow X$ be an isometry of $X$.
        Then exactly one of the following occurs:
        \begin{enumerate}
         \item $\gamma$ is \emph{elliptic}: For any $x \in X$ the $\langle \gamma \rangle$-orbit of $x$ is bounded in $X$.
         \item $\gamma$ is \emph{hyperbolic/loxodromic}: The induced homeomorphism $\partial \gamma: \partial X \rightarrow \partial X$ has precisely two fixed points $\gamma^+, \gamma^- \in \partial X$. For any $x \in X$, the $\langle \gamma \rangle$-orbit map $\mathbb{Z} \rightarrow X$ given by $n \rightarrow \gamma^n x$ is a quasi-isometric embedding, and $$\lim_{n \rightarrow \pm \infty} \gamma^n x = \gamma^{\pm}.$$
         \item $\gamma$ is \emph{parabolic}: The induced homeomorphism $\partial \gamma: \partial X \rightarrow \partial X$ has precisely one fixed point $\gamma^+$, and 
         $$ \lim_{n \rightarrow \pm \infty} \gamma^n x = \gamma^+.$$
        \end{enumerate}
  \end{prop}

 \subsection{The weak convex hull} \label{sec:gos_via_hulls}
 
  Using the notion of a weak convex hull introduced by Swenson \cite{swenson}, we define a subspace of a model geometry for $G \in \cC$ for each vertex and edge of the JSJ tree that will play the role of a coarse graph of spaces decomposition in Section~\ref{sec:cubulation}.
 
  \begin{defn} \label{def:wch}
    Let $X$ be a proper geodesic hyperbolic metric space. The {\it weak convex hull} of a set $A \subset \partial X$, denoted $\textrm{WCH}_X(A)$, is the union of all geodesic lines in $X$ which have both endpoints in $A$. Given a subset $S \subseteq X$ let $\Lambda S = \overline{S} \cap \partial X$ where $\overline{S}$ denotes the closure of $S$ in $X \cup \partial X$.     
     If $H \leq \isom(X)$ then $\Lambda H = \Lambda (H x)$ where $x \in X$. The space $\Lambda H$ does not depend on the choice of $x$.
  \end{defn}
    
  \begin{thm}\cite[Main Theorem]{swenson} \label{thm:pdc_on_hull}
   Let $G$ act properly and cocompactly by isometries on $X$.
   If $H$ be a quasi-convex subgroup of $G$, then $H$ acts properly and cocompactly on $\textrm{WCH}_X(\Lambda H) \subseteq X$.
  \end{thm}
 
    \section{Construction of a common cubulation} \label{sec:cubulation}
   
  \subsection{Background: Quasi-actions on trees}
  
   \begin{defn}
      Let $G$ be a group, let $X$ be a metric space, and let $K \geq 1$, $C \geq 0$. A {\it $(K,C)$-quasi-action} of $G$ on $X$ is a map $G \times X \rightarrow X$ denoted $(g,x) \mapsto A_g(x) = g\cdot x$, so that for each $g \in G$, the map $A_g:X \rightarrow X$ is a $(K,C)$-quasi-isometry of $X$, and for each $x \in X$ and $g,h \in G$, the distance between $A_g \circ A_h$ and $A_{gh}$ in the sup-norm is uniformly bounded independent of $g,h \in G$.  
      A quasi-action is {\it cobounded} if there exists a constant $R$ such that for each $x \in X$, the spaces $G \cdot x$ and $X$ are within Hausdorff distance of $R$ from each other. 
      A quasi-action is {\it proper} if for each $R$ there exists $M$ so that for all $x,y \in X$, the cardinality of the set $\{g \in G \, | \, (g\cdot N_R(x)) \cap N_R(y) \neq \emptyset\}$ is at most $M$.
      If a group $G$ quasi-acts on metric spaces $X$ and $Y$, a {\it quasi-conjugacy from $X$ to $Y$} is a quasi-isometry $f:X \rightarrow Y$ which is {\it coarsely $G$-equivariant}, meaning that $d_Y(f(g \cdot x), g \cdot fx)$ is uniformly bounded independent of $g \in G$ and $x \in X$. A tree $T$ is {\it bushy} if every point of $T$ is a uniformly bounded distance from a vertex which has at least three unbounded complementary components.
     \end{defn}
    
    \begin{lemma} (Quasi-action principle.) \label{lem:quasi-action-principle} \cite{moshersageevwhyteI}
    If a group $G$ acts by isometries on a metric space $X$ and $X$ is quasi-isometric to a metric space $X'$, then $G$ quasi-acts on the metric space~$X'$. If a finitely generated group $G$ is quasi-isometric to a metric space $X$, then there is a cobounded and proper quasi-action of $G$ on $X$. 
   \end{lemma}
  
   \begin{thm} \cite[Theorem 1]{moshersageevwhyteI} \label{thm:quasi_action}
    If $G \times T' \rightarrow T'$ is a quasi-action of a group $G$ on a bounded valence bushy tree $T'$, then there is a bounded valence bushy tree $T$, an isometric action $G \times T \rightarrow T$, and a quasiconjugacy $f:T \rightarrow T'$ from the action of $G$ on $T$ to the quasi-action of $G$ on $T'$. 
   \end{thm}
   
   We will also need the following lemma, which can be deduced from well-known results; see \cite{paulin}. We include a proof for the benefit of the reader.  
   
   \begin{lemma} \label{lem:equiv_b_map}
    If a group $G$ acts by isometries on proper geodesic hyperbolic metric spaces $X$ and $Y$, then a quasi-conjugacy from the action of $G$ on $X$ to the action of $G$ on $Y$ induces a $G$-equivariant homeomorphism $\p X \rightarrow \p Y$. 
   \end{lemma}

    \begin{proof}
     Let $f:X \rightarrow Y$ be a quasi-conjugacy between the $G$-actions. Since $f$ is a quasi-isometry between hyperbolic metric spaces, $f$ induces a homeomorphism from $\p X$ to $\p Y$. To prove this map is $G$-equivariant, for $g \in G$, let $\iota_g:X \rightarrow X$ and $\iota_g':Y \rightarrow Y$ be the isometries given by the $G$-action on $X$ and $Y$, respectively. By definition of quasi-conjugacy, the maps $f \circ \iota_g$ and $\iota_g' \circ f$ commute up to uniformly bounded distance. Let $\rho$ be a quasi-geodesic ray in $X$. Since $X$ and $Y$ are hyperbolic, the image of $\rho$ under both maps is a quasi-geodesic ray in $Y$. Since the maps commute up to bounded distance, the images of $\rho$ under $f \circ \iota_g$ and $\iota_g' \circ f$  are at bounded Hausdorff distance, hence define the same point on the boundary as desired.
    \end{proof}
 
   \subsection{Construction of an action on a tree} \label{subsec:constructingAnAction}
   
   \begin{notation} \label{nota:wch_decomp}
      Suppose that $G, G' \in \mathcal{C}_k$ act properly and cocompactly on a proper geodesic metric space $X$. We use the graph of groups notation set in Remark~\ref{remark:gr_of_gps} and Remark~\ref{remark:same_JSJ_tree}. In particular, as described in Remark~\ref{remark:same_JSJ_tree}, there is a homomorphism $\Phi_T: \isom(X) \rightarrow \isom(T)$ that defines actions of $G$ and $G'$ on $T$. For $v \in VT$, let $G_v \leq G$ and $G_v' \leq G'$ denote the stabilizer of $v$ in $G$ and $G'$, respectively. For $e \in ET$, let $G_e \leq G$ and $G_e' \leq G'$ denote the stabilizer of $e$ in $G$ and $G'$, respectively. 
      
      Since $G$ and $G'$ act transitively on vertices of valence $k$ in $T$, we may assume there is a vertex $v_0 \in VT$ of valence $k$ so that $G_{v_0} = G_{0}$ and $G_{v_0}' = G_{0}'$. Let $v_1, \ldots, v_k$ be the (infinite valence) vertices adjacent to $v_0$. Without loss of generality, $G_{v_i} = G_i$ and $G_{v_i}' = G_i'$ for $1 \leq i \leq k$. 
      
      For each vertex $v \in VT$ and edge $e \in ET$, let 
      \[ X_v = \textrm{WCH}_X(\Lambda G_v) = \textrm{WCH}_X(\Lambda G_v')\quad \textrm{ and } \quad X_e = \textrm{WCH}_X(\Lambda G_e) = \textrm{WCH}_X(\Lambda G_e'),\]
      where the equalities $\textrm{WCH}_X(\Lambda G_v) = \textrm{WCH}_X(\Lambda G_v')$ and $\textrm{WCH}_X(\Lambda G_e) = \textrm{WCH}_X(\Lambda G_e')$ follow from the construction of the JSJ tree given by Bowditch \cite{bowditch}. (In the language of \cite{bowditch}, $\Lambda G_v$ is a {\it necklace} if $v$ has infinite valence and a {\it jump} if $v$ has finite valence.) If $u \in VT$ is a vertex of finite valence, we refer to the subspace $X_u \subset X$ as a {\it peripheral subspace}.       
      By Theorem~\ref{thm:pdc_on_hull}, the groups $G_i$ and $G_i'$ act properly and cocompactly on the space $X_{v_i}$ for $i \in \{0, \ldots, k\}$.
      
      There is a cyclic ordering on $\p X_{v_i}$ for $i \in \{1, \ldots, k\}$ given by the topology of $\p X$. Choose an orientation on the cyclic ordering of the Cantor set $\p X_{v_i}$ as in Definition~\ref{def:op}. Let $\isom^o(X_{v_i})$ denote the isometries of $X_{v_i}$ that preserve the oriented cyclic order on $\p X_{v_i}$.       
      By Proposition~\ref{prop:bow_cyclic_order}, for $i \in \{1, \ldots, k\}$ the groups $G_i$ and $G_i'$  preserve the cyclic ordering on $\partial X_{v_i}$. Moreover, since $G_i$ and $G_i'$ are the fundamental groups of oriented surfaces with boundary, these actions are orientation-preserving.
      Suppose these faithful actions are denoted by homomorphisms $\Phi_i:G_i \rightarrow \isom^{o}(X_{v_i})$ and $\Phi_i':G_i' \rightarrow \isom^{o}(X_{v_i})$.

      For $i \in \{1, \ldots, k \}$, let 
      \[ H_{i} = \big\la \; \Phi_{i}(G_i), \; \Phi_{i}'(G_i') \big\ra \leq \isom^o(X_{v_i}). \] 
      The groups $G_i$ and $G_i'$ embed in $H_i$, and in an abuse of notation, we refer to $\Phi_i(g_i) \in H_i$ and $\Phi_i(g_i')$ as $g_i$ and $g_i'$, where $g_i$ and $g_i'$ are elements defined in Remark~\ref{remark:gr_of_gps}.
     \end{notation}

   \begin{lemma} \label{lemma:H_i_action}
    The group $H_i$ has an action by isometries on a locally-finite bushy tree $T_i$ that is quasi-conjugate to the action of $H_i$ on $X_{v_i}$. There exists a bi-infinite geodesic line $A_i \subset T_i$ so that $g_i$ and $g_i'$ act by non-trivial translation along~$A_i$. 
   \end{lemma}
    \begin{proof}      
      The fundamental group of a hyperbolic surface with boundary acts properly and cocompactly on~$X_{v_i}$. Therefore, there is a quasi-isometry $F_i':X_{v_i} \rightarrow T_i'$, where $T_i'$ is a locally-finite bushy tree. Since $H_i$ acts by isometries on $X_{v_i}$,  Lemma~\ref{lem:quasi-action-principle} yields a quasi-action of $H_i$ on $T_i'$.       
      Therefore, by Theorem \ref{thm:quasi_action}, there is a locally-finite bushy tree $T_i$, an isometric action of~$H_i$ on $T_i$, and a quasi-conjugacy $f_{i}:T_i \rightarrow T_i'$ from the action of~$H_i$ on $T_i$ to the quasi-action of~$H_i$ on $T_i'$. Let $F_{i}:T_i' \rightarrow T_i$ be the quasi-inverse of $f_{i}$. Then, $F_{i} \circ F_{i}':X_{v_i} \rightarrow T_i'$ is a quasi-conjugacy.     
      By Lemma \ref{lem:equiv_b_map}, $F_{i} \circ F_{i}'$ induces a $H_i$-equivariant homeomorphism $(F_{i} \circ F_{i}')_{\p}:\p X_{v_i} \rightarrow \p T_i$. 
      Since $\la g_i, g_i' \ra \leq H_i$ stabilizes $X_{v_0} \subseteq X_{v_i}$, the group $\la g_i, g_i' \ra$ stabilizes $\p X_u = \{\alpha_i, \beta_i\} \subseteq \p X_{v_i}$.  Let $A_i \subseteq T_i$ be the unique geodesic connecting these points. Since $F_{i} \circ F_i'$ is a quasi-conjugacy, $g_i$ and $g_i'$ act on $A_i$ by non-trivial translations. 
     \end{proof}

     \begin{remark}
      After possibly replacing $T_i$ with $\textrm{WCH}_{T_i}(\partial T_i)$, we may assume that $T_i$ has no vertices of valence one.
      In particular, if an isometry of $T_i$ fixes $\partial T_i$, then the isometry is the identity.
     \end{remark}

\subsection{Properties of the action}
  
  \begin{defn}
    Let $\cB_i$ denote the set of peripheral subspaces of $X_{v_i}$ as in Notation~\ref{nota:wch_decomp}. Let $\p \cB_i = \{ \partial X_u \subseteq \p X_{v_i} \mid X_u \in \cB_i\}$. Each peripheral subspace is $2$-ended, and each element in $\partial \cB_i$ is a cut pair in $\p G$ as in Proposition~\ref{prop:bow_cyclic_order}. (In the language of~\cite{bowditch}, $\partial \cB_i$ is the set of \emph{jumps} in $\partial X_{v_i}$ and is determined by the topology of $\partial X$.) 
    Moreover, the group $H_i$ leaves the set $\p \cB_i$ invariant.  By Lemma~\ref{lem:equiv_b_map} and Lemma~\ref{lemma:H_i_action}, there exists an $H_i$-equivariant homeomorphism $\phi_i : \p X_{v_i} \rightarrow \p T_i$.  For $b \in \partial \cB_i$, let $\overline{b}$ be the unique bi-infinite geodesic in $T_i$ connecting the two points in $\p (\phi_i (b))$. Let $\overline{\p \cB_i} =\{ \overline{b} \, | \, b \in \partial\cB_i\} \subset T_i$, which we call the set of {\it peripheral lines of $T_i$}.
   \end{defn}
  
    Equip $\p T_i$ with the oriented cyclic order induced by the homeomorphism $\phi_i : \p X_{v_i} \rightarrow \p T_i$. Let $\isom^o(T_i)$ denote the set of isometries of $T_i$ that preserve this oriented cyclic order. Let $\Psi_{i} : H_i \rightarrow \isom(T_i)$ denote the action of $H_i$ on $T_i$ and let $\overline{H}_{i}$ denote $\Psi_{i}(H_i)$. Since $\phi_i$ is $H_i$-equivariant, $\overline{H}_i \leq \isom^o(T_i)$. In addition, there exists a cyclic order on the lines in $\oB$ which is preserved by $H_i$. 
  
  \begin{defn}
    An isometry of $\mathbb{R}$ is \emph{orientation preserving} if it does not exchange the ends at infinity.
    \end{defn}

    An isometry in $\isom^o(T_i)$ that stabilizes a line $\overline{b} \in \oB$ restricts to an orientation preserving isometry of $\overline{b}$, as exchanging the ends of $\overline{b}$ would induce a non-orientation-preserving homeomorphism of $\partial T_i$.

    \begin{lemma} \label{lemma_fix_line_id}
      If $h \in \isom^o(T_i)$ fixes a line $\overline{b} \in \oB$, then $h$ is the identity isometry of $T_i$. 
    \end{lemma}
    \begin{proof}
      Suppose $h \in \isom^o(T_i)$ fixes $\overline{b} \in \oB$. Then, by Proposition~\ref{prop:elliptic}, $h$ is an elliptic isometry of $T_i$. To prove $h$ fixes $T_i$, it is enough to prove $h$ fixes $\p T_i$. Suppose towards a contradiction $h \cdot x = y$ for some $x \neq y \in \p T_i$. Since $\{ \p \overline{b} \, | \, \overline{b} \in \oB\} \subset \p T_i$ is dense, there exists $\overline{b'} \in \oB$ so that $h \cdot \overline{b'} \neq \overline{b'}$. There is a cyclic order on $\p T_i$ preserved by $h$ and $h \cdot b = b$, so $h^2 \cdot b' \neq b'$. Similarly, $h^n \cdot b' \neq b'$ for all $n \in \N$. Let $x \in b'$. The vertex and edge groups of $G$ and $G'$ have the bounded packing property by Lemma~\ref{lemma:bounded_packing}, so the set $\oB$ satisfies the bounded packing property in $T_i$. So, the diameter of the orbit of $x$ in $T_i$ is unbounded, contradicting the fact that $h$ is an elliptic isometry of~$T_i$. Thus, $h$ fixes $\p T_i$, and, since $T_i$ does not have vertices of valence one, $h$ fixes~$T_i$. 
    \end{proof}

   \begin{lemma} \label{lemma_h_i_proper}
   The group $\isom^o(T_i)$ acts on $T_i$ properly. 
  \end{lemma}
  \begin{proof} 
     Since $T_i$ is a locally-finite bushy tree, to prove $\isom^o(T_i)$ acts on $T_i$ properly, it is enough to prove that $\stab(x) \leqslant \isom^o(T_i)$ is finite for all $x \in VT_i$, the vertex set of $T_i$. Let $x \in VT_i$. The vertex and edge groups of $G$ and $G'$ have the bounded packing property by Lemma~\ref{lemma:bounded_packing}, so the set $\oB$ satisfies the bounded packing property in $T_i$. Therefore, there exists $r > 0$ such that the subset $\overline{\cB}_x \subseteq \oB$ containing all elements that non-trivially intersect $N_r(x)$ is finite and contains at least two elements. 
     As $\stab(x)$ stabilizes the subset $\overline{\cB}_x$, there is a homomorphism $\phi: \stab(x) \rightarrow S_{2n}$ where $n = |\overline{\cB}_x|$ and $S_{2n}$ denotes the symmetric group on $2n$ elements: the homomorphism is given by the action of $\stab(x)$ on the $2n$ boundary points of the elements in $\overline{\cB}_x$.
     If $h$ is in the kernel of this homomorphism, then $h$ fixes every element in $\overline{\cB}_x$. So, by Lemma~\ref{lemma_fix_line_id}, the element $h$ is trivial. Hence, the kernel of $\phi$ must be trivial. Therefore, $\stab(x)$ is finite and $\isom^o(T_i)$ acts properly on $T_i$.
   \end{proof}

    The following is an immediate consequence of Lemma~\ref{lemma:subgroup} and Lemma~\ref{lemma_h_i_proper}.

   \begin{corollary} \label{cor:finite_Index}
    The groups $\Phi_i(G_i)\cong G_i$ and $\Phi_i'(G_i')\cong G_i'$ are finite-index subgroups of~$\overline{H}_i$. Hence, $\Phi_i(G_i) \cap \Phi_i'(G_i') \leq \overline{H}_i \leq \isom^o(T_i)$ is a finite-index subgroup of both $\Phi_i(G_i)$ and $\Phi_i'(G_i')$.
   \end{corollary}

    \subsection{Commensurability of translation numbers}	\label{sec:comm_tran_num}

      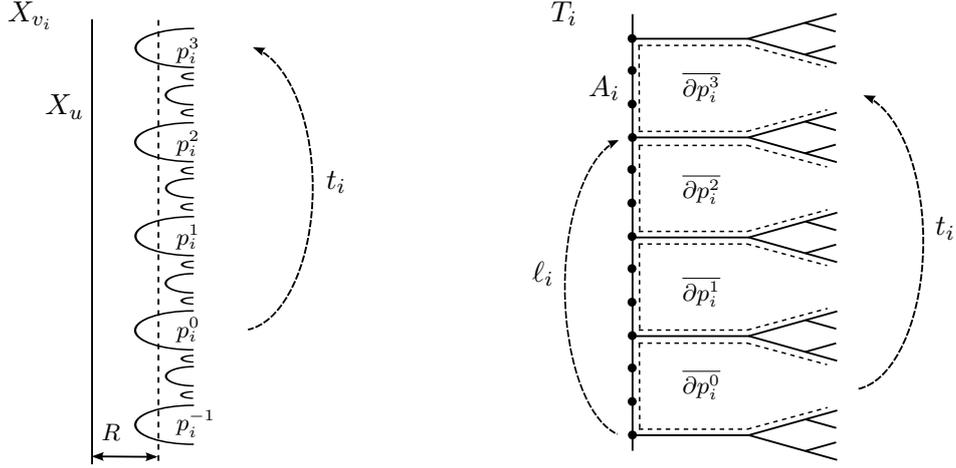
\begin{figure}
        \label{figure:tr_length}
	\begin{overpic}[width=.8\textwidth, tics=5]{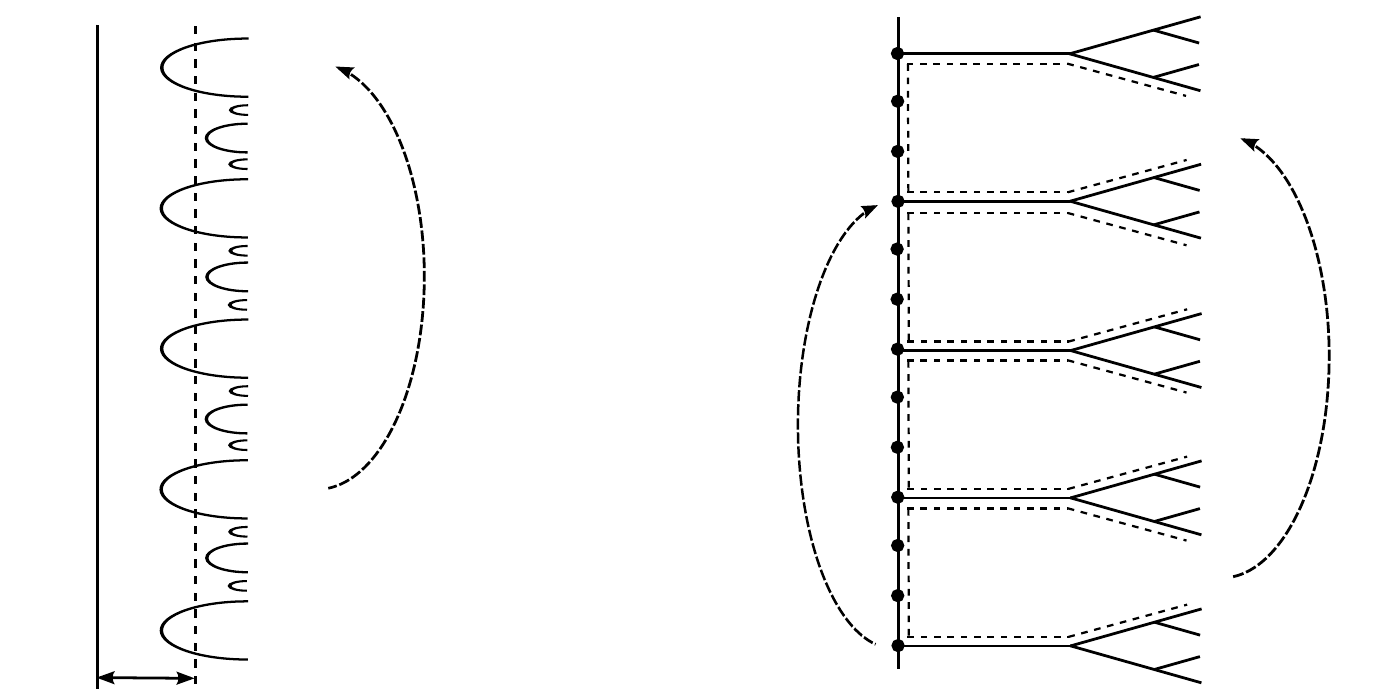}
	 \put(-2,48){$X_{v_i}$}
	 \put(2,38){$X_u$}
	 \put(8,3){\Small{$R$}}
	 \put(16,4){\Small{$p_i^{-1}$}}
	 \put(16,14){\Small{$p_i^0$}}
	 \put(16,24){\Small{$p_i^1$}}
	 \put(16,34){\Small{$p_i^2$}}
	 \put(16,44){\Small{$p_i^3$}}
	 \put(32,30){$t_i$}
	 \put(56,48){$T_i$}
	 \put(60,40){$A_i$}
	 \put(54,20){$\ell_i$}
	 \put(97,25){$t_i$}
	 \put(70,8){\Small{$\overline{\p p_i^0}$}}
	 \put(70,18){\Small{$\overline{\p p_i^1}$}}
	 \put(70,29){\Small{$\overline{\p p_i^2}$}}
	 \put(70,40){\Small{$\overline{\p p_i^3}$}}
	\end{overpic}
	\caption{\small{ An illustration of the peripheral subspaces of $X_{v_i}$ and $T_i$ and the two notions of translation length.}}
    \end{figure}
 
    \begin{notation} \label{nota:tr_len}
	The groups $G_i$ and $G_i'$ embed as finite-index subgroups in $\overline{H}_{i}$, so $G_i$ and $G_i'$ act properly and cocompactly on $T_i$. Let $\ell_i$ and $\ell_i'$ be the translation lengths of $g_i$ and $g_i'$ in $T_i$, respectively, where $g_i$ and $g_i'$ are the elements defined in Remark~\ref{remark:gr_of_gps}. The translation lengths of $g_i, g_i'$ are realized on $A_i \subset T_i$. See Figure~3.
    \end{notation}
   
   The next proposition shows the two vectors that record translation lengths are equivalent up to scalar multiplication by integers. 
   
    \begin{prop} \label{prop:comm_tr_len}
      There exist positive integers $L, L' \in \N$ so that \[L(\ell_1, \ldots, \ell_k) = L'(\ell_1', \ldots, \ell_k').\] 
    \end{prop}

   To prove the proposition, we introduce another notion of translation number. 
   
   \begin{notation} 
     Let $R >0$ be large enough such that $N_R(X_u) \subset X$ non-trivially intersects at least one peripheral subspace in each $X_{v_i}$.
     Let $\mathcal{P}_i = \{ p \in \mathcal{B}_i \mid p \cap N_R(X_u) \neq \emptyset \}$.
     Note that $\mathcal{P}_i$ is stabilized by both $G_0$ and $G_0'$, where $G_0$ and $G_0'$ are defined in Remark~\ref{remark:gr_of_gps}.
     By Proposition~\ref{prop:bow_cyclic_order}, for each $i$, there is a cyclic order on the set of all peripheral subspaces in $X_{v_i}$ and this order is preserved by any isometry of $X$ which stabilizes $X_{v_i}$.
     Thus, there is an indexing $\mathcal{P}_i = \{p_i^r \mid r \in \mathbb{Z}\}$ such that $p_i^r \leq p_i^s$ if and only if $r \leq s$. See Figure~3. 
   \end{notation}

    \begin{defn}
     For each $i$,  $1 \leq i \leq k$, define the {\it $i^{th}$ peripheral translation length} of $g_0$ and $g_0'$, by $t_i$ and $t_i'$ respectively, where 
     \[ g_0 \cdot p_i^0= p_i^{t_i} \, \textrm{ and }\, g_0' \cdot p_i^0 = p_i^{t_i'}.\]
    \end{defn}

    Assume that $t_i, t_i'\in \Z - \{0\}$ are positive, and observe that
    \begin{eqnarray} \label{eqn:mult}
     g_0^n \cdot p_i^0 = p_i^{nt_i} \, &\textrm{ and }& \, 
     (g_0')^n \cdot p_{i}^{0} = p_i^{nt_i'}.
    \end{eqnarray}
    
    \begin{lemma} \label{lemma:comm_of_per_len}
      There exist non-zero integers $K, K'>0$ so that \[K(t_1, \ldots, t_k) = K'(t_1', \ldots, t_k').\] 
    \end{lemma}
    \begin{proof}
      To prove the lemma, we show 
      $t_1't_i = t_1t_i'$ for all $i$, $1 \leq i \leq k$. 
      Towards a contradiction, suppose $t_1't_i \neq t_1t_i'$ for some $i$. 
      Let 
      \[ h = g_0^{t_1'}(g_0')^{-t_1}.\]
      By Equation (\ref{eqn:mult}) above, 
      \[ h \cdot p_i^0 = p_i^{t_1't_i - t_1t_i'} .\]
      In particular, $h^n$ stabilizes $p_1^0$, but $h^n$ has non-trivial $i^{th}$ peripheral translation length for $i \in \{2, \ldots, k\}$.

      Let $\pi : X \rightarrow X_u$ be the closest-point projection map.  The image $\pi(p_i^0) \subset X_u$ is contained in a bounded subset $Q_i \subset X_u$ for $i \in \{1,\ldots, k\}$. The subset $Q_1$ is fixed by $h$ since $h$ stabilizes~$p_1^0$. However, $d(Q_i, h^nQ_i) \rightarrow \infty$ for $i \in \{2, \ldots, k\}$, since $h$ has non-trivial $i^{th}$ peripheral translation length for $i \in \{2, \ldots, k\}$ and the elements of $\cP_i$ satisfy the bounded packing property. Thus, $d(h^nQ_1, h^nQ_2) \rightarrow \infty$, a contradiction. 
      \end{proof}
 
      \begin{proof}[(Proof of Proposition \ref{prop:comm_tr_len})]
       We prove $t_i'\ell_i = t_i\ell_i'$ for each $i \in \{1, \ldots, k\}$ and apply Lemma~\ref{lemma:comm_of_per_len}. See Figure~3.
       Let $\partial \mathcal{P}_i \subseteq \partial \mathcal{B}_i \cong \p X_{v_i}$ be the image of the map $\psi:\mathcal{P}_i \rightarrow \p X_{v_i}$ defined by $\psi(p_i^r) = \partial p_i^r$. The map $\psi$ is $G_0$-equivariant and $G_0'$-equivariant. 
       
       Let $\overline{\p \mathcal{P}_i} = \{ \overline{\partial p_i^r}\, |\, p_i^r \in \cP_i \} \subseteq \oB \subseteq T_i$. Since $\psi$ is $G_0$-equivariant and $G_0'$-equivariant,
       \[
         g_0 \cdot \overline{\partial p_i^0} = \overline{\partial p_i^{t_i}} \, \textrm{ and } \, g_0' \cdot \overline{\partial p_i^0} = \overline{\partial p_i^{t_i'}}.
       \]     
       Towards a contradiction, suppose that $\ell_it_i' \neq \ell_i't_i$ for some $i \in \{1, \ldots, k\}$. 
       Let $$ h_i = g_0^{t_i'}(g_0')^{-t_i}.$$
       Then, $$h_i \cdot (\overline{\partial p_i^0}) = (\overline{\partial p_i^0}),$$
       but, $h_i$ has translation length $t_i' \ell_i - t_i \ell_i' \neq 0$ in $T_i$.
       Let $a \in A_i$ so it belongs in the axis of $h_i$, where $A_i$ is defined in Lemma~\ref{lemma:H_i_action}.
       Then,
       $$d_{T_i}\big(h_i^n \cdot a, h_i^{n}\cdot \overline{\partial p_i^0} \big) = d_{T_i}\big(h_i^n \cdot a, \overline{\partial p_i^0} \big) \rightarrow \infty,$$
       a contradiction since $h_i$ is an isometry of $T_i$.
         \end{proof}
     
    \subsection{Construction of a cubulation} \label{subsec:cubulation}

      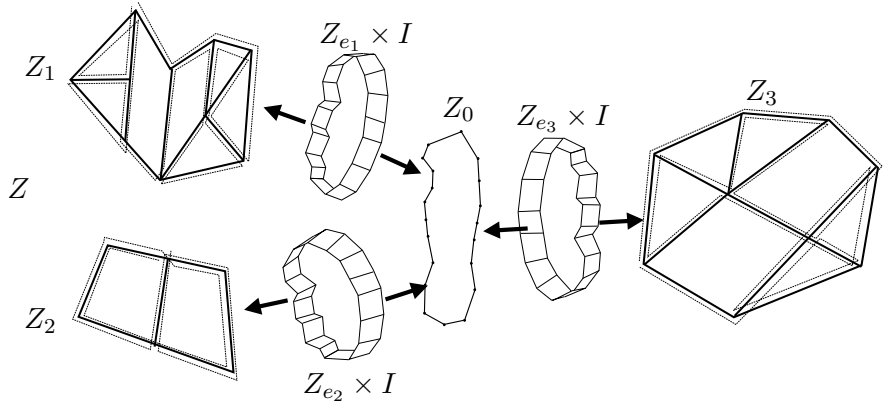
\begin{figure}
        \label{figure:cubulation}
	\begin{overpic}[width=.7\textwidth,tics=10]{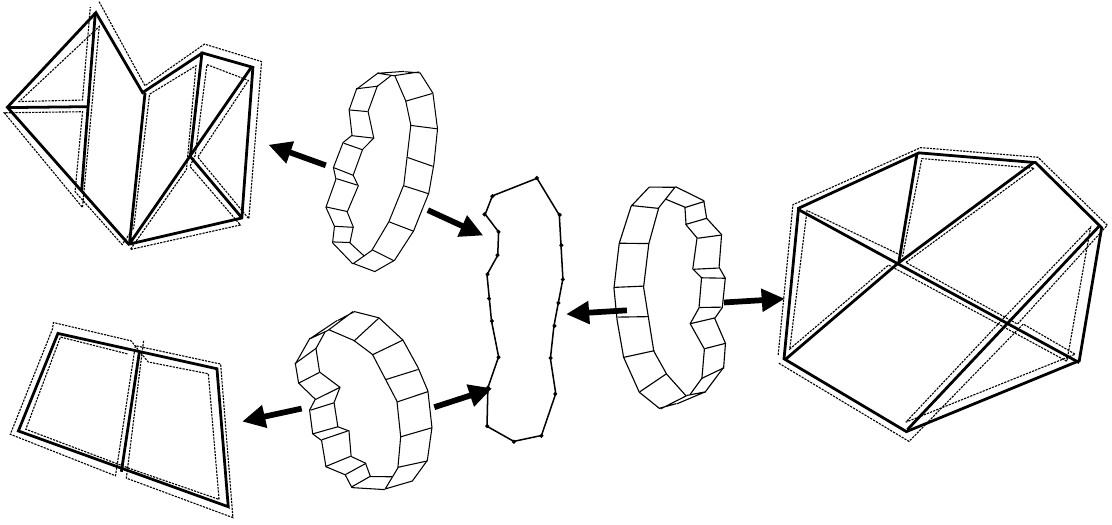}
	 \put(-7,23){$Z$}
	 \put(-5,38){$Z_{1}$}
	 \put(-5,7){$Z_{2}$}
	 \put(83,35){$Z_{3}$}
	 \put(46,33){$Z_{0}$}
	 \put(31,42){$Z_{e_1}\times I$}
	 \put(29,-1){$Z_{e_2} \times I$}
	 \put(55,32){$Z_{e_3} \times I$}
	\end{overpic}
	\caption{A new graph of spaces decomposition $Z$ for $\cG$, the JSJ decomposition of $G$. The spaces $Z_i$ are finite graphs, and the black arrows indicate the attaching maps. Dotted lines indicate the closed immersed paths $\omega_i$ and $I$ denotes the interval $[-1,1]$.}
    \end{figure}
 
   In this section we construct a new common model geometry for $G$ and $G'$.
   We do this by constructing graphs of spaces $Z$ and $Z'$ such that $G = \pi_1(Z)$ and $G' = \pi_1(Z')$. 
   We then give an isomorphism between their universal covers $F : \widetilde{Z} \rightarrow \widetilde{Z}'$. 
   There may be many possible isomorphisms between the universal covers of $Z$ and $Z'$, but we will construct our isomorphism to address the following issue with the existing actions of $G$ and $G'$ on $X$ and $T$: although $T / G$ and $T / G'$ are isomorphic graphs, the $G$-orbits of vertices in $T$ and the $G'$-orbits of vertices in $T$ may not coincide as subsets.
   That is to say, $T / \langle G, G' \rangle$ may be a proper quotient of $T / G$.
   We will construct the isomorphism $F$ to fix this.
   The induced map $F_*: T \rightarrow T$ will send the $G$-orbit of a vertex to a $G'$-orbit of a vertex (see the equality $q = q' \circ F_*$ in the statement of Theorem~\ref{thm:the_isomorphism}).
   This is of vital importance in ensuring that we can embed $G$ and $G'$ as finite-index subgroups in the color-preserving automorphism group in Theorem~\ref{thm:NewAction}.

  \begin{construction} \label{const:common_cube}
      We construct  non-positively curved cube complexes $Z$ and $Z'$ that are graphs of spaces for $\cG$ and $\cG'$, the JSJ decompositions of $G$ and $G'$, respectively. An illustration appears in Figure~4.
  
      As a consequence of Lemma~\ref{lemma:H_i_action}, the groups $G_i$ and $G_i'$ act freely and cocompactly on $T_i$, a locally-finite bushy tree.  Fix a basepoint $\tilde{a}_{i}$ in $A_i \subset T_i$. Orient the line $A_i$; as in Definition~\ref{def:op}.
      This orientation is preserved by both $G_i$ and $G_i'$. 
      Let $Z_i = T_i/G_i$ and let $Z_i' = T_i /G_i'$. 
     Then $Z_i$ and $Z_i'$ are finite graphs such that $G_{i} =\pi_1 (Z_i, a_{i})$ and $G_{i}' = \pi_1(Z_i', a_{i}')$ where $a_{i}$ and $a_{i}'$ are the images of $\tilde{a}_{i}$ in the respective quotients. 
     
     Let $\omega_i \subset Z_i$ and $\omega_i' \subset Z_i'$ denote the image of $A_i$ in the quotient of $T_i$ under the action of $G_i$ and $G_i'$, respectively. Then $\omega_i$ and $\omega_i'$ are closed immersed paths of length $\ell_i$ and $\ell_i'$, respectively, where $\ell_i$ and $\ell_i'$ are defined in Notation~\ref{nota:tr_len}. Orient the paths $\omega_i$ and $\omega_i'$ so that they  lift to the same orientation on $A_i \subset T_i$.
     Subdivide each edge in $Z_i$ and $Z'_i$ into $\frac{\ell}{\ell_i}$ edges, where $\ell = \lcm(\ell_1, \ldots, \ell_k, \ell_1', \ldots, \ell_k')$. Then, after subdivision, the length of $\omega_i$ is equal to $\ell$, and the length of $\omega_i'$ is equal to $\frac{\ell\ell_i'}{\ell_i}$ for $1 \leq i \leq k$. By Proposition~\ref{prop:comm_tr_len}, there exist $L,L' \in \N$ so that $\frac{\ell_i'}{\ell_i} = \frac{L}{L'}$ for all $i = 1, \ldots, k$. Thus, the length of $\omega_i'$ is equal to $\frac{\ell L}{L'}$ for all $i = 1, \ldots, k$. The universal covers $\widetilde{Z}_i$ of $Z_i$ and $\widetilde{Z}_i'$ of $Z_i'$ are both isomorphic to the tree $T_i$ with each edge subdivided into $\frac{\ell}{\ell_i}$ edges. 
      
     Build a non-positively curved cube complex $Z$ that is a graph of spaces for the JSJ decomposition $\cG$ of $G$ given in Remark~\ref{remark:gr_of_gps} as follows.     
     Let the pointed vertex space $(Z_{u_0},z_{u_0})$ be an oriented circle constructed out of $\ell$ many $1$-cubes and $z_{u_0}$ be a zero cube in $Z_{u_0}$. For $i =1, \ldots, k$, let the pointed vertex space $(Z_{u_i},z_{u_i})$ be $(Z_{i}, a_i)$. 
     For each edge $e_i = (u_0,u_i) \in E\G$, let the pointed edge space $(Z_{e_i}, z_{e_i})$ be an oriented circle constructed out of $\ell$ many $1$-cubes and $z_{e_i}$ a $0$-cube on $Z_{e_i}$. The attaching map $\theta^-_{e_i} : (Z_{e_i},z_{e_i}) \rightarrow (Z_{u_0}, z_{u_0})$ is an isomorphism of cube complexes that is orientation-preserving. The attaching map $\theta^+_{e_i} : (Z_{e_i},z_{e_i}) \rightarrow (Z_{u_i}, z_{u_i}) = (Z_i,a_i)$ for $i \in \{1, \ldots, k\}$ is the cubical map sending $(Z_{e_i}, z_{e_i})$ to the based loop $(\omega_i, a_i)$, which also has length $\ell$, so that the map preserves orientation. The space $Z$ is a non-positively curved square complex since the attaching maps $\theta_i^{\pm}$ are locally isometric embeddings.     
      
     Construct $Z'$ in an analogous fashion. The main difference in this case is the vertex space $(Z_{u_0}',z_{u_0}')$ is a circle constructed out of $\frac{\ell L}{L'}$ many $1$-cubes and the edge spaces are also built out of $\frac{\ell L}{L'}$ many $1$-cubes. Then, there exist cubical attaching maps to the curves $\omega_i' \subset Z_i'$ of length $\frac{\ell L}{L'}$.      
     \end{construction}
        
     \begin{assumption}
      For the remainder of the paper, we assume each edge in $T_i$ has been subdivided into $\frac{\ell}{\ell_i}$ edges for $i \in \{1, \ldots, k\}$. 
     \end{assumption}
   
    \begin{notation} \label{nota_q}
      The universal covers $\widetilde{Z}$ and $\widetilde{Z}'$ of the cube complexes $Z$ and $Z'$ built in Construction~\ref{const:common_cube} have a natural tree of spaces decomposition. Indeed, view $\G$ and $T$ as $CW$-complexes. Let $q,q': T \rightarrow \Gamma$ denote the quotient maps from the JSJ tree $T$ to the underlying graph~$\G$ of the JSJ decompositions of $G$ and $G'$, respectively. There are maps $p: \widetilde{Z} \rightarrow T$ and $p':\widetilde{Z}' \rightarrow T$ as follows. There is a map $\bp:Z \rightarrow \Gamma$ defined by mapping each vertex space $Z_{u_i}$ to the vertex $u_i$ for $i = 0, \ldots, k$ and mapping $Z_{e_i} \times [-1,1]$ to $[-1,1] \cong e_i$ for $1 \leq i \leq k$. The map $p$ is a lift of the map $\bp$ to the universal cover $\widetilde{Z}$ which makes the left-hand diagram below commute, where $\pi:\widetilde{Z} \rightarrow Z$ is the covering map.
       \begin{displaymath}
	\xymatrix{
	      \widetilde{Z} \ar[d]^{\pi} \ar[r]^p & T \ar[d]^{q} & & \widetilde{Z}' \ar[d]^{\pi'} \ar[r]^{p'} & T \ar[d]^{q'}\\
	      Z \ar[r]^{\bp} & \G & & Z' \ar[r]^{\bp'} & \G
		  }
      \end{displaymath}
      If $v \in VT$, then let $\widetilde{Z}_v$ denote $p^{-1}(v)$, and for $e \in ET$ let $\widetilde{Z}_e \times [-1,1]$ denote the closure of the $p$-preimage of the interior of $e$. The map $p':\wZ' \rightarrow T$ and the tree of spaces decomposition of $\widetilde{Z}'$ is similar, and we add $'$ to denote the analogous objects.  
      Let $e = (u,v) \in ET$ so that $u$ has finite valence and $v$ has infinite valence. Suppose $q(v) = u_i$ and $q'(v) = u_j$ for some $i,j \in \{1, \ldots, k\}$. Let $\widetilde{\theta}_e^-:\widetilde{Z}_e \rightarrow \widetilde{Z}_u$ denote a lift of the attaching map $\theta_{e_i}^-:Z_{e_i} \rightarrow Z_{u_0}$, and let $\widetilde{\theta}_e^+:\wZ_e \rightarrow \wZ_v$ denote a lift of the attaching map $\theta_{e_i}^+:Z_{e_i} \rightarrow Z_{u_i}$. Similarly, let $\wt_{e}^{'-}:\wZ_e' \rightarrow \wZ_u'$ denote a lift of the attaching map $\theta_{e_j}':Z_{e_j}' \rightarrow Z_{u_0}'$, and let $\wt_{e}^{'+}:\wZ_e' \rightarrow \wZ_v'$ denote a lift of the attaching map $\theta_{e_j}':Z_{e_j}' \rightarrow Z_{u_j}'$.      
      
      If $\wZ$ and $\wZ'$ are isomorphic cube complexes, then an isomorphism $F: \wZ \rightarrow \wZ'$ induces an isomorphism $F_{*}:T \rightarrow T$ determined by the tree of spaces structure of $\wZ$ and $\wZ'$ described above. 
    \end{notation}  
   
  \begin{thm} \label{thm:the_isomorphism}
  	There is an isomorphism $F: \widetilde{Z} \rightarrow \widetilde{Z}'$ such that the induced map $F_* : T \rightarrow T$ satisfies the equality $q(v) = q' \circ F_*(v)$.
  \end{thm}

  \begin{proof}  
      We exhibit an isomorphism $F:\wZ \rightarrow \wZ'$ by giving isomorphisms $F_v:\wZ_v \rightarrow \wZ'_{v'}$ for all $v \in VT$ and $F_e: \wZ_e \rightarrow \wZ_{e'}'$ for all $e \in ET$, where $F_*(u) = u' \in VT$ and $F_*(e) = e' \in ET$ are specified in the construction. 
      The union of these maps defines an isomorphism $F$ provided the following diagrams commute for all $e = (-e,+e) \in ET$ with $F_*(e) = e' = (-e', +e')$.
     
      \begin{displaymath}
      \xymatrix{ 
    	\widetilde{Z}_{-e} \ar[rr]^{F_{-e}}  &  & \widetilde{Z}_{-e'}'   
    	& &  \widetilde{Z}_{+e} \ar[rr]^{F_{+e}}  &  & \widetilde{Z}_{+e'}'   \\  
    	\widetilde{Z}_{e} \ar[rr]^{F_e} \ar[u]^{\widetilde{\theta}_{e}^-} & & \widetilde{Z}_{e'}' \ar[u]^{\widetilde{\theta'}_{e'}^-} 
    	& & \widetilde{Z}_{e} \ar[rr]^{F_e} \ar[u]^{\widetilde{\theta}_{e}^+} & & \widetilde{Z}_{e'}' \ar[u]^{\widetilde{\theta'}_{e'}^+} \\
    	} \label{diagram:commutes}  \tag{$\ast$}
      \end{displaymath}
  
      A key technical point in the inductive proof we are about to provide is that the $2$-ended vertex spaces in $\widetilde{Z}$ and $\widetilde{Z}'$ need to be consistently aligned, as in Figure~5.
      To ensure this, for all $u$ such that $q(u) = u_0$ and $u' = F_*(u)$ we specify basepoints $\tilde{z}_u, \tilde{z}_{u'}'$,  such that $\pi(\tilde{z}_u) = z_{u_0}$, $\pi'(\tilde{z}_u') = z_{u_0}'$, and $F_u(\tilde{z}_u) =  \tilde{z}_{u'}'$.
  
      \begin{figure}
	\begin{overpic}[scale=.8,tics=10]{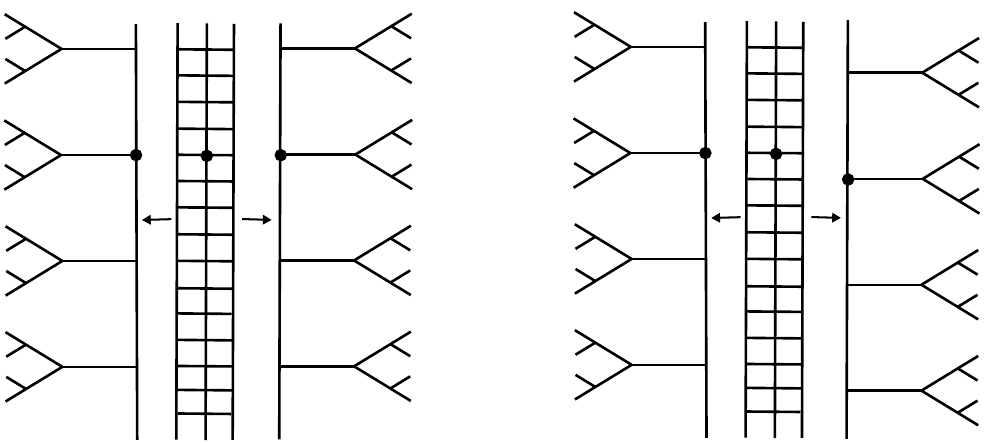} 
	 \end{overpic}
		 \caption{If the diagram on the left arises from the space $\widetilde{Z}$ and the diagram on the right arises from the space $\widetilde{Z'}$, then the 2-ended vertex spaces are misaligned and $\widetilde{Z}$ and $\widetilde{Z}'$ are not isomorphic. }
	 \label{figure:misalign}
   \end{figure}

      We construct $F$ inductively. The base case is obtained as follows. Recall from Notation~\ref{nota:wch_decomp} there exists a vertex $v_0 \in VT$ of valence $k$ such that $q(v_0)=q'(v_0)=u_0 \in V\G$, and if $v_1, \ldots, v_k \in VT$ are the vertices adjacent to $v_0$, then $q(v_i) = q'(v_i) = u_i \in V\G$.  
      The base case is constructing $F_{v_0}$. 
      By construction, $\wZ_{v_0}$ and $\wZ'_{v_0}$ are isomorphic to $\R$ equipped with the standard cube complex structure and an orientation. 
      Let $\wz_{v_0}$ and $\wz'_{v_0}$ be $0$-cubes in $\wZ_{v_0}$ and $\wZ'_{v_0}$, respectively, that project to the basepoints $z_{u_0}$ and $z'_{u_0}$ in $Z_{u_0}$ and $Z'_{u_0}$, respectively. 
      Let $F_{v_0}:(\wZ_{v_0}, \wz_{v_0}) \rightarrow (\wZ'_{v_0}, \wz'_{v_0})$ be a cubical isomorphism that is orientation preserving. 
      
      The inductive assumption is that $F$ has been successfully defined over all vertices and edges in the subtree $\Omega_n = \overline{N}_{2n}(v_0) \subseteq~T$.
      Each vertex $u$ at distance $2n$ from $v_0$ has $q(u) = u_0$ so there are basepoints $\tilde{z}_u$ and $\tilde{z}_{F(u)}$ such that $F_u(\tilde{z}_u) = \tilde{z}_{F(u)}$.
      
        \begin{figure}
      	\begin{overpic}[width=.65\textwidth, tics=5]{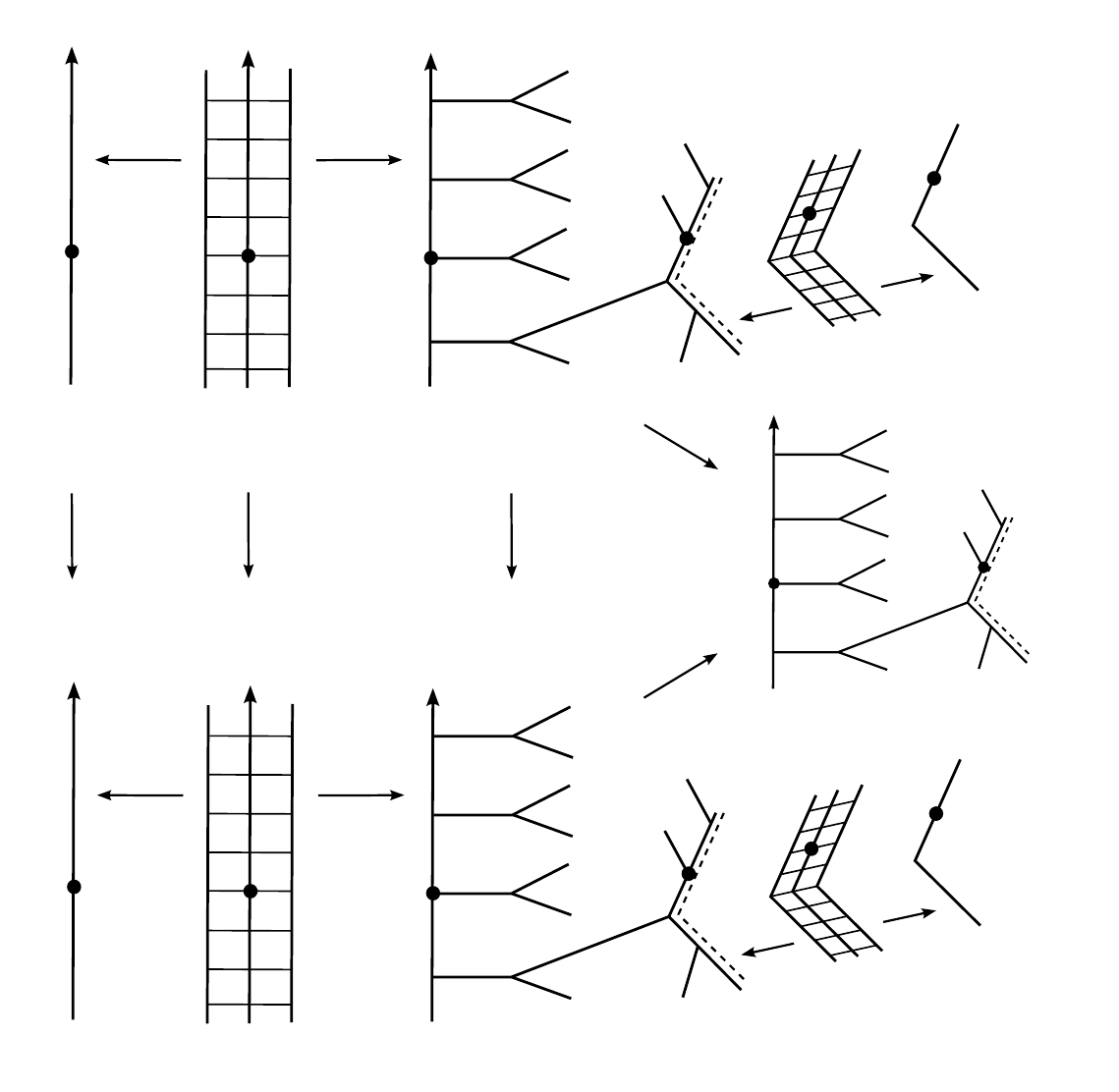}
	\put(5,59){\Small{$\wZ_u$}}
	\put(20.5,59){\Small{$\wZ_e$}}
	\put(43,59){\Small{$\wZ_v$}}
	\put(7.5,49){\Small{$F_u$}}
	\put(23.5,49){\Small{$F_e$}}
	\put(46.5,49){\Small{$F_v$}}
	\put(2.5,75){\Small{$\wz_u$}}
	\put(19,75.5){\Small{$\wz_e$}}
	\put(34.5,75){\Small{$\wz_v$}}
	\put(11,85){\Small{$\widetilde{\theta}_e^-$}}
	\put(31,85){\Small{$\widetilde{\theta}_e^+$}}
	\put(5,2){\Small{$\wZ'_{u'}$}}
	\put(20.5,2){\Small{$\wZ'_{e'}$}}
	\put(43,2){\Small{$\wZ'_{v'}$}}
	\put(1.5,17.5){\Small{$\wz'_{u'}$}}
	\put(18.8,18.7){\Small{$\wz'_{e'}$}}
	\put(34,17.5){\Small{$\wz'_{v'}$}}
	\put(11,29){\Small{$\widetilde{\theta}_{e'}^{'-}$}}
	\put(31,29){\Small{$\widetilde{\theta}_{e'}^{'+}$}}
	\put(59,54){\Small{$f_v$}}
	\put(59,39.5){\Small{$f'_{v'}$}}
	\put(83,57){\Small{$T_i$}}
	\put(65.5,45){\Small{$\widetilde{a}_i$}}
	\put(65,57){\Small{$A_i$}}
	\put(93,40){\Small{$\ell$}}
	\put(90,46){\Small{$a_{\ell}$}}
	\put(75,65){\Small{$\wZ_{\hat{e}}$}}
	\put(86,67){\Small{$\wZ_w$}}
	\put(85,80){\Small{$\wz_w$}}
	\put(65,71){\Small{$\wt_{\hat{e}}^+$}}
	\put(78,74){\Small{$\wt_{\hat{e}}^-$}}
	\put(75,7.5){\Small{$\wZ'_{\hat{e}'}$}}
	\put(86,10){\Small{$\wZ'_{w'}$}}
	\put(85,24){\Small{$\wz'_{w'}$}}
	\put(64.7,14.2){\Small{$\wt_{\hat{e}'}^{'+}$}}
	\put(77.7,17.2){\Small{$\wt_{\hat{e}'}^{'-}$}}
	\end{overpic}
	 \label{figure:induction}
	\caption{\small{An illustration of the maps used in the inductive step of the construction.}}
    \end{figure}
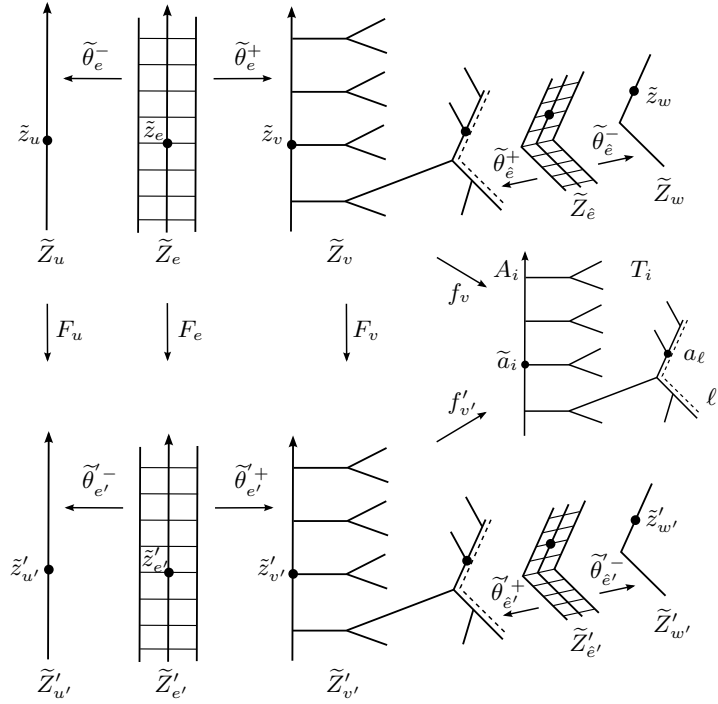
      
      The inductive step requires extending $F$ from $\Omega_n$ to $\Omega_{n+1}$. 
      An illustration of the construction appears in Figure~6. 
      First, suppose $u \in V\Omega_n \subset VT$ is a vertex at distance $2n$ from $v_0$, and let $u'$ be the vertex such that $\widetilde{Z}_u$ is mapped to $\widetilde{Z}_{u'}'$ by $F_u$.
      Thus $q(u)=q'(u') = u_0$. 
      If $e=(u,v)$ is an edge incident to $u$, then there is a unique edge $e' = (u', v')$ such that $q(v) = q'(v') = u_i$ for some $i \in \{1, \ldots, k\}$. 
      This correspondence defines a bijection between the $k$ vertices adjacent to $u$ and the $k$ vertices adjacent to $F_{*}(u)$. 
      Suppose the edge $e=(u,v)$ is not contained in the subtree $\Omega_n$. 
      By our inductive assumption on $\Omega_n$, there exists basepoints $\widetilde{z}_u \in \wZ_u$ and $\widetilde{z}'_{u'} \in \wZ_{u'}'$ such that $F_u(\widetilde{z}_u) = \widetilde{z}'_{u'}$ and $\widetilde{z}_u$ and $\widetilde{z}_{u'}'$ project to the basepoints $z_{u_0} \in Z_{u_0}$ and $z_{u_0}' \in Z_{u_0}'$. 
      Let $\wz_e = (\wt_e^-)^{-1}(\wz_u) \in \wZ_e$ and $\wz_{e'}' = (\wt_{e'}^{'-})^{-1}(\wz'_{u'}) \in \wZ_{e'}'$.      
      Define 
      \[ F_e = \wt^{'-}_{e'} \circ F_u \circ (\wt_e^-)^{-1}: (\wZ_e, \widetilde{z}_e) \rightarrow (\wZ_{e'}', \widetilde{z}'_{e'}). \]
      The map $F_e$ is an orientation-preserving isomorphism since all three maps in its composition are orientation-preserving isomorphisms. By construction, Diagram~(*) commutes.

      Define an isomorphism $F_v:\wZ_v \rightarrow \wZ_{v'}'$ as follows. Let $\wz_v = \wt_e^+(\wz_e)$ and $\wz_{v'}' = \wt_{e'}^{'+}(\wz_{e'}')$. Then, since $\wt_e^+$ and $\wt_{e'}^{'+}$ are lifts of the attaching maps $\theta_{e_i}^+$ and $\theta_{e_i}^{'+}$, the $0$-cubes $\wz_v$ and $\wz_{v'}'$ project to the basepoints $z_{u_i} \in Z_{u_i}$ and $z_{u_i}' \in Z'_{u_i}$, respectively. Therefore, since $q(v) = q'(v') = u_i$, there are isomorphisms 
      \[f_v: (\wZ_v, \wt_e^+(\wZ_e), \wz_v) \rightarrow (T_i, A_i, \widetilde{a}_i) \]
      and 
      \[f'_{v'}: (\wZ'_{v'}, \wt_{e'}^{'+}(\wZ'_{e'}), \wz'_{v'}) \rightarrow (T_i, A_i, \widetilde{a}_i) \]
      obtained from lifts of the isomorphisms $(Z_{u_i}, z_{u_i}) \rightarrow (T_i/G_i, a_i)$, and $(Z_{u_i}', z_{u_i}') \rightarrow (T_i/G_i', a_i')$.
      These lifts respect the cyclic orderings on the boundary of the respective spaces, and respect the orientation on the lines $\wt_e^+(\wZ_e)$, $\wt_{e'}^{'+}(\wZ'_{e'})$, and $A_i$. 
      Thus, the map
      \[F_v = (f'_{v'})^{-1} \circ f_v:(\wZ_v, \wt_e^+(\wZ_e), \wz_v) \rightarrow (\wZ'_{v'}, \wt_{e'}^{'+}(\wZ'_{e'}), \wz'_{v'})  \]
      is an isomorphism, preserving orientation on $\wt_e^+(\wZ_e)$ and $\wt_{e'}^{'+}(\wZ'_{e'})$ and for which Diagram~(*) commutes. 
      Therefore, by extending $F$ along all such edges and vertices we extend $F$ from $\Omega_n = \overline{N}_{2n}(v_0)$ to $\overline{N}_{2n+1}(v_0)$.
      
      For the second step, suppose we continue to consider $u$ at distance $2n$ from $v_0$, an adjacent vertex $v$ at distance $2n+1$ from $v_0$, and let  $w \in VT$ be a vertex incident to $v$ and at distance $2n+2$ from $v_0$. 
      Let $\hat{e} = (w,v) \in ET$. Let $\cA_v = f_v^{-1}(\oB)$ and let $\cA'_{v'} = (f'_{v'})^{-1}(\oB)$, where $f_v$ and $f'_{v'}$ are the isomorphisms defined in the preceding paragraph. The sets $\cA_v$ and $\cA'_{v'}$ are the peripheral lines in $\wZ_v$ and $\wZ'_{v'}$, respectively, and are in one-to-one correspondence with the vertices adjacent to $v$ and $v'$, respectively. Thus, there is an edge $\hat{e}' = (v',w')$ so that $F_v(\wt_{\hat{e}}^+(\wZ_{\hat{e}})) = \wt_{\hat{e}'}^{'+}(\wZ'_{\hat{e}'})$. 
      Then $f_v(\wt_{\hat{e}}^+(\wZ_{\hat{e}})) = f'_{v'}(\wt_{\hat{e}'}^{'+}(\wZ'_{\hat{e}'})) = \ell$, a peripheral line in the tree $T_i$. 
      By Corollary~\ref{cor:finite_Index}, there is a $0$-cube $a_{\ell} \in \ell$ in the orbit of $\widetilde{a}_i$ under the action of $\Phi_i(G_i) \cap \Phi_i'(G_i')$ on $T_i$. 
      In particular, $f_v^{-1}(a_{\ell})$ and $(f'_{v'})^{-1}(a_{\ell})$ project to the basepoints $z_{u_i}\in Z_{u_i}$ and $z'_{u_i} \in Z'_{u_i}$, respectively. Furthermore, $F_v(f_v^{-1}(a_{\ell})) = (f'_{v'})^{-1}(a_{\ell})$.
      Let $\wz_{\hat{e}} = (\wt_{\hat{e}}^+)^{-1}(f_v^{-1}(a_{\ell}))$ and $\wz'_{\hat{e}'} = (\wt^{'+}_{\hat{e}'})^{-1}((f'_{v'})^{-1}(a_{\ell}))$. 
      Define
      \[F_{\hat{e}} = \wt_{\hat{e}'}^{'+} \circ F_v \circ  (\wt_{\hat{e}}^+)^{-1}:(\wZ_{\hat{e}},\wz_{\hat{e}}) \rightarrow (\wZ'_{\hat{e}'}, \wz'_{\hat{e}'}). \]
      Let $\wz_w = \wt_{\hat{e}}^-(\wz_{\hat{e}})\in \wZ_w$ and $\wz'_{w'} = \wt_{\hat{e}'}^{'-}(\wz'_{\hat{e}'}) \in \wZ'_{w'}$. Let
      \[F_w = \wt_{\hat{e}'}^{'-} \circ F_{\hat{e}} \circ (\wt_{\hat{e}}^-)^{-1} : (\wZ_w, \wz_w) \rightarrow (\wZ'_{w'}, \wz'_{w'}). \]
      Then $F_{\hat{e}}$ and $F_w$ are isomorphisms so that Diagram~(*) commutes and the induction hypotheses are satisfied.             
      
      Therefore, the map $F$ extends from $\Omega_n$ to $\Omega_{n+1}$. As $T$ is the ascending union of $\Omega_n$, we conclude by induction there is an isomorphism $F:\wZ \rightarrow \wZ'$ as desired.  
  \end{proof}
  \section{Rigid edge colorings}

   \begin{defn}
      Let $G \in \cC_k$, let $X$ be a model geometry for~$G$, and let $T$ be the JSJ tree for~$G$. A {\it rigid coloring} of $T$ is a map $c:ET \rightarrow \{1,\ldots, k\}$ such that edges adjacent to a finite-valence vertex have different colors and edges adjacent to an infinite-valence vertex have the same color. If $G \leq \isom(T)$, we say a rigid coloring $c$ of $T$ is {\it $G$-invariant} if $c(e) = c(g\cdot e)$ for all $e \in ET$ and $g \in G$. Given a rigid coloring $c$ of $T$, let $\Aut_c(X)$ be the isometries of $X$ that induce a color-preserving map of $T$.  
   \end{defn}
   
   Suppose $G,G' \in \cC_k$ have a common model geometry $X$ with JSJ tree $T$. Define a rigid coloring $c$ of $T$ by $c(e) = i$ if $q(e) = e_i$, where $q:T \rightarrow \G$ is defined in Notation~\ref{nota_q}. Then, the coloring $c$ of $T$ is $G$-invariant, but is not necessarily $G'$ invariant. Theorem~\ref{thm:the_isomorphism} yields a rigid coloring of $T$ that is both $G$-invariant and $G'$-invariant. 
  
  \begin{thm} \label{thm:NewAction}
   If $G, G' \in \cC_k$ act properly and cocompactly on a proper geodesic metric space~$X$, then $G, G'$ act properly and cocompactly on a CAT(0) square complex $\mathcal{X}$ such that the following properties hold. 
   \begin{enumerate}
    \item There is a map $p: \cX \rightarrow T$ that is both $G$-equivariant and $G'$-equivariant, and $p$ determines a tree of spaces structure of $\cX$.
    \item If $u \in VT$ has infinite valence, then the vertex space $\cX_u = p^{-1}(u)$ is a tree. If $u \in VT$ has finite valence, then the vertex space $\cX_u = p^{-1}(u)$ is a copy of $\R$.
    \item If $e \in ET$, then $\cX_e \times [-1,1] = p^{-1}(e)$ is isomorphic to $\mathbb{R} \times [-1,1]$.  
    \item There is a rigid coloring $c: ET \rightarrow \{1, \ldots, k\}$ that is $G$-invariant and $G'$-invariant.
   \end{enumerate}
   Moreover, $G$ and $G'$ are finite-index subgroups of $\textrm{Aut}_c(\mathcal{X})$.   
  \end{thm}
  \begin{proof}
  By Construction~\ref{const:common_cube}, $G$ is the fundamental group of a non-positively curved square complex $Z$ that is a graph of spaces for the JSJ decomposition $\cG$ of $G$, and $G'$ is the fundamental group of a non-positively curved square complex $Z'$ that is a graph of spaces for the JSJ decomposition $\cG'$ of $G'$.
  By Theorem~\ref{thm:the_isomorphism}, there is an isomorphism $F: \widetilde{Z} \rightarrow \widetilde{Z}'$ such that the induced map $F_*:T \rightarrow T$ satisfies $q' \circ F_*(u) = q(u)$ for all $u\in VT$.
  Therefore, $G'$ acts on $\widetilde{Z}$ by conjugating the action of $G'$ on $\widetilde{Z}'$ by $F$.
  The map $q$ determines a $G$-invariant coloring of $T$ by $c(e) = i$ if $q(u) = e_i$.
  As $q'\circ F_* = q$, the action of $G'$ on $\widetilde{Z}$ also preserves the coloring $c$. Hence, $G,G' \leq \Aut_c(\cX)$, where $\cX = \wZ$. Moreover, Properties~(1)-(4) hold by construction. 
  
  Each vertex space $\mathcal{X}_v$ for $q(v) = v_i$ with $i \in \{1, \ldots, k\}$ is isomorphic to the tree $T_i$. The actions of $G_v$ and $G_v'$ on $\mathcal{X}_v$ correspond to the actions of $G_v$ and $G_v'$ on $T_i$ constructed in Subsection~\ref{subsec:constructingAnAction}.
  The peripheral lines in $\overline{\partial \mathcal{B}_i}$ inside of $T_i$ correspond precisely to the intersections $(\mathcal{X}_e \times [-1,1]) \cap \mathcal{X}_v$ where $e$ is an edge adjacent to $v$.
  As $G$ is quasi-isometric to~$\mathcal{X}$, we deduce $\partial{\cX} \cong \partial{G}$, and by Proposition~\ref{prop:bow_cyclic_order}, the cyclic order on $\partial \mathcal{X}_v$ is given by the cyclic order on~$\p T_i$.
  Therefore, we may apply Lemma~\ref{lemma_fix_line_id} and Lemma~\ref{lemma_h_i_proper} to the subgroup of isometries in $\textrm{Aut}_c(\mathcal{X})$ that stabilize $\mathcal{X}_v$.
  
  To prove the final statement, we first show if an isometry $\gamma \in \Aut_c(\cX)$ fixes a vertex space pointwise, then $\gamma$ is trivial. Indeed, suppose first that $\gamma$ fixes a vertex space $\cX_v$ with $q(v) = u_i$ for some $i\in \{1,\ldots, k\}$. Then, the peripheral lines in $\cX_v$ are fixed, hence the incident edge spaces, which are homeomorphic to $\R \times [-1,1]$, are fixed. Thus, the $2$-ended vertex spaces adjacent to $\cX_v$ are fixed. So, suppose $\gamma$ fixes a $2$-ended vertex space $\cX_u$. The isometry $\gamma$ induces a map on $T$ preserving the rigid coloring of $T$, and each of the finitely many edges incident to $u$ has a unique color. Thus, the action of $\gamma$ on $T$ must fix these edges, and hence $\gamma$ must fix the corresponding edge spaces. Then, by Lemma~\ref{lemma_fix_line_id}, $\gamma$ is the identity on each adjacent infinite-ended vertex space. This argument may be continued, so $\gamma$ is the identity on $\cX$ if $\gamma$ fixes a vertex space. 
  
  Finally, let $x$ be a $0$-cube in $\cX$ contained in a vertex space $\cX_v$; to prove $\Aut_c(\cX)$ acts on $\cX$ properly, it is enough to prove that the stabilizer $\stab(x) \leq \Aut_c(\cX)$ is finite. 
  If $\cX_v$ is homeomorphic to $\mathbb{R}$, then there is an index-2 subgroup of $\stab(x)$ that fixes $\cX_v$.
  By the previous paragraph, this index-2 subgroup must be trivial.
  If $\cX_v$ is homeomorphic to a tree, then there exists $r>0$ such that the set $\cE$ of edge spaces which intersect the neighborhood $N_r(x)$ is finite and contains at least two elements. 
  There is an action of $\stab(x)$ on $\cE$; hence, there is a homomorphism $\stab(x) \rightarrow S_{2n}$, where $n = |\cE|$ and $S_{2n}$ denotes the symmetric group on $2n$-elements: the homomorphism is given by the action of $\stab(x)$ on the $2n$ boundary points of the elements in $\cE$. 
  By Lemma~\ref{lemma_fix_line_id}, the kernel of this homomorphism fixes $\cX_v$ and by the preceding paragraph must be trivial.
  Thus, $\stab(x)$ is finite.
    
  Therefore, $\textrm{Aut}_c(\mathcal{X})$ acts properly and cocompactly on $\mathcal{X}$. Hence, $\Aut_c(\cX)$ contains both $G$ and $G'$ as finite-index subgroups by Lemma~\ref{lemma:subgroup}.
  \end{proof}
 
  \section{Proof of the main theorem} \label{sec:main_thm}
 
    To prove that commensurable groups in $\cC$ act on a common model space, we use the abstract commensurability classification of groups in the class $\mathcal{C}$. The following theorem was shown in \cite{stark} for $k=4$, and easily extends to arbitrary $k$; see also \cite{danistarkthomas}.

      \begin{thm} \cite[Theorem 3.3.3]{stark} \label{thm:acclasses}
	Let $Y, Y' \in \mathcal{Y}_k$ be the unions of $k \geq 3$ surfaces $\Sigma_1, \ldots, \Sigma_k$ and $\Sigma_1', \ldots, \Sigma_k'$, respectively, where each surface has one boundary component, all boundary components of the $\Sigma_i$ are identified, and all boundary components of the $\Sigma_i'$ are identified. Let $G \cong \pi_1(Y)$ and let $G' \cong \pi_1(Y')$. Then $G$ and $G'$ are abstractly commensurable if and only if there exist non-zero integers $L, L' \in \Z$ so that, up to permuting the indices of the $\Sigma_i'$, \[L(\chi(\Sigma_1), \ldots, \chi(\Sigma_k)) = L'(\chi(\Sigma_1'), \ldots, \chi(\Sigma_k')).\] 
      \end{thm}
    
    Finally, we collect the results above to prove the main theorem of the paper. 
    
     \begin{thm} 
    Let $G, G' \in \mathcal{C}_k$. The following are equivalent.
      \begin{enumerate}
	 \item The groups $G$ and $G'$ act properly and cocompactly by isometries on the same proper geodesic metric space.
	 \item The groups $G$ and $G'$ are abstractly commensurable.
	 \item There exists a group $\mathfrak{G}$ that contains $G$ and $G'$ as finite-index subgroups.
      \end{enumerate}
  \end{thm} 
    \begin{proof}
     We first show conditions (2) and (3) are equivalent. Suppose $G$ and $G'$ are abstractly commensurable. By Theorem~\ref{thm:acclasses}, the groups $G$ and $G'$ are finite-index subgroups of the same group $\mathfrak{G}$, where $\mathfrak{G} \cong \pi_1(Y_0)$ and $Y_0$ is the union of surfaces whose Euler characteristics have no common divisor. Thus, (2) implies (3), and clearly, (3) implies~(2). 
     
     We now show conditions (1) and (3) are equivalent. Suppose first that condition (3) holds. Then any model geometry for $\mathfrak{G}$ is a common model geometry for $G$ and $G'$, proving (1). If condition (1) holds, then condition (3) holds by Theorem~\ref{thm:NewAction}.
    \end{proof}
    
    \begin{remark} \label{remark:RACG}
     Groups in the class $\mathcal{C}_k$ are quasi-isometric to certain right-angled Coxeter groups, including those with defining graph (and nerve) a planar graph called a {\it $3$-convex generalized $\Theta$-graph}; see \cite{danithomas,davis},\cite[Definition 1.6]{danistarkthomas} for definitions and background. If $W_{\Lambda}$ is such a right-angled Coxeter group, then the JSJ decomposition of $W_{\Lambda}$ is similar to the JSJ decomposition of a group in $\cC_k$ as in Remark~\ref{remark:gr_of_gps}; in particular, the JSJ decomposition of $W_{\Lambda}$ has underlying graph $\G$ constructed for some $k \in \N$. This JSJ decomposition has one two-ended vertex group $\la v,w \, | \, v^2= w^2 = 1 \ra \cong D_{\infty}$, the infinite dihedral group. Thus, to apply the results in Section~\ref{sec:cubulation}, one considers the infinite-order element $vw \in W_{\Lambda}$ and its image in each of the maximally hanging Fuchsian vertex groups. The actions of the maximally hanging Fuchsian vertex groups on the trees $T_i$ guaranteed by Lemma~\ref{lemma:H_i_action} are not free, rather properly discontinuous, so the quotients are orbi-complexes, which can be glued together along homeomorphic $1$-dimensional suborbifolds. Theorem~\ref{thm:NewAction} extends to this setting. The abstract commensurability classification of hyperbolic right-angled Coxeter groups with defining graph a generalized $\Theta$-graph is given in \cite[Theorem 1.8]{danistarkthomas}, and one can deduce Theorem~\ref{maintheorem} holds for this class of groups as well. 
    \end{remark}

\bibliographystyle{alpha}
\bibliography{ACQI}

\end{document}